\documentclass{article}
\date{}
\usepackage[utf8]{inputenc}
\usepackage[T1]{fontenc}
\usepackage[a4paper]{geometry}
\usepackage{babel}
\usepackage{lmodern}
\usepackage{tikz}
\usetikzlibrary{arrows,shapes,positioning,shadows,trees,matrix}

\usepackage{amsmath}
\usepackage{amsfonts}
\usepackage{amssymb}
\usepackage{enumitem}
\usepackage{mdframed}
\usepackage[colorlinks=true]{hyperref}
\usepackage[all]{xy}
\usepackage{array}
\usepackage{mathrsfs}
\usepackage{alltt}
\usepackage{fancyhdr}
\usepackage{authblk}
\usepackage{multicol}
\usepackage{cite}
\usepackage{dsfont}
\usepackage[standard]{ntheorem}
\theoremstyle{break}
\newtheorem{Theo}{Theorem}[section]
\newtheorem{Def}[Theo]{Definition}

\newtheorem{Ex}[Theo]{Example}
\newtheorem{Prop}[Theo]{Proposition}  

\theoremstyle{plain}

\newtheorem{rk}[Theo]{Remark}
\newtheorem{lemme}[Theo]{Lemma}            

\theoremstyle{empty}
\newtheorem{cond}{Condition}
\usepackage[S]{thmbox}

\theoremstyle{break}
\newtheorem[M,thickness=2pt]{Theoperso}[Theo]{Theorem}
\newtheorem[M,thickness=2pt]{theoalt}{Theorem}[Theo]
\newtheorem[M,thickness=2pt]{Propperso}[Theo]{Proposition}  
\newtheorem[S,thickness=1pt]{lemmeperso}[Theo]{Lemma}  
\newtheorem[M,thickness=2pt]{Theof}[Theo]{Théorème}

\newtheorem[M,thickness=2pt]{corperso}[Theo]{Corollary}
\newenvironment{theop}[1]{
  \renewcommand\thetheoalt{#1}
  \theoalt
}{\endtheoalt}

\newcommand{\N}{\mathbb{N}}
\newcommand{\R}{\mathbb{R}}

\newcommand{\C}{\mathbb{C}}
\renewcommand{\O}{\mathcal{O}}

\newcommand{\Q}{\mathbb{Q}}

\newcommand{\G}{\mathbb{G}}

\newcommand{\B}{\mathbb{B}}

\renewcommand{\P}{\mathbb{P}}

\newcommand{\dr}{\partial}
\newcommand{\mc}{\mathcal}

\newcommand{\clog}[2]{\Omega_{#1}(\log #2)}

\DeclareMathOperator{\grass}{Grass}

\DeclareMathOperator{\rg}{rk}
\DeclareMathOperator{\supp}{Supp}
\DeclareMathOperator{\Res}{Res}

\DeclareMathOperator{\Aut}{Aut}
\DeclareMathOperator{\bs}{Bs}

\DeclareMathOperator{\mult}{mult}
\newcommand{\un}{\mathds{1}}
\title{Hyperbolicity of smooth logarithmic and orbifold pairs in $\P^n$}
\author[,1,2]{Clara Dérand \thanks{clara.derand@univ-lorraine.fr \\claraderand@orange.fr \\ ORCID iD 0009-0004-5163-8837}}
\affil[1]{Université de Lorraine, CNRS, IECL, Vand\oe uvre-lès-Nancy, France}
\affil[2]{Université de Lorraine, CNRS, CRAN, Vand\oe uvre-lès-Nancy, France}

\begin{document}
\maketitle
\begin{abstract}
    We derive a necessary and sufficient condition on a hyperplane arrangement in $\P^n$ for the associated logarithmic cotangent bundle to be ample modulo boundary. We extend this result to the orbifold setting and give some applications concerning hyperbolicity of pairs. We improve significantly the results of \cite{DR20}.
\end{abstract}
\section{Introduction}
\subsection{Hyperbolicity of quasi-projective varieties}
It is now well-known that hyperbolicity properties of a (quasi-)projective variety can be 
investigated via the positivity of its (logarithmic) cotangent bundle. This is pictured by the Fundamental vanishing lemma, which under the following form first appeared in a paper by Noguchi \cite{Nog77}.
\begin{Theo}[Fundamental vanishing lemma]
\label{vanishing}
    Let $X$ be a smooth projective variety and $D\subset X$ a simple normal crossings divisor. Let $f:\C\to X\setminus D$ be a non-constant holomorphic map and denote by $\tilde{f}:\C\to\P(\clog{X}{D})$ its canonical lift. Then for any ample line bundle $A$ on $\P(\clog{X}{D})$ and any section $\widetilde\omega\in H^0(\P(\clog{X}{D}),\O_{\P(\clog{X}{D})}(m)\otimes A^{-1})$, one has \[\tilde f^*\widetilde \omega\equiv0.\]

    In other words, for any entire curve $f:\C\to X\setminus D$, one has \[\tilde f(\C)\subset \B_+(\O_{\P(\clog{X}{D})}(1)).\]
\end{Theo}
For this reason, projective varieties with ample cotangent bundle has received much attention: indeed, in this case, the theorem above with $D=\emptyset$ immediately shows that $X$ is hyperbolic.

Let $(X,D)$ be a smooth log pair. It has been shown by several authors that hyperbolicity properties of $X\setminus D$ could be investigated through the positivity of the associated logarithmic cotangent bundle $\clog{X}{D}$, see e.g. \cite{No86}, \cite{BD18}, \cite{BD19}, \cite{DR20},\cite{CDR20},\cite{CDDR21}.

In the non-compact case, though, it appears that the logarithmic cotangent bundle is never ample. Actually, one finds trivial quotients supported on the different components of $D$. Each such quotient gives a positive dimensional subvariety of the augmented base locus $\B_+(\O_{\P(\clog{X}{D}}(1))$. 

More recently, many works also focused on the logarithmic setting, and there are now several examples for which the logarithmic cotangent bundle is big, see \emph{e.g.} \cite{No86},\cite{Rou07},\cite{Rou09}, \cite{BD18}, \cite{BD19}, \cite{DR20},\cite{CDR20},\cite{CDDR21}. 

Nevertheless, if one asks for a stronger positivity property, there are up to now very few results, essentially by \cite{Noc83}, \cite{BD18}, \cite{DR20} and \cite{ADT22}.
A natural question to ask is indeed whether the trivial quotients mentioned above are the only obstructions to ampleness. 
In \cite{BD18}, the notion of \emph{almost ampleness } was introduced to express this minimality of the logarithmic cotangent bundle.
\begin{Def}
Let $(X,D)$ be a smooth log pair, with $D=\sum_{i=1}^c D_i$. 
The logarithmic cotangent bundle $\clog{X}{D}$ is said almost ample if it satisfies
\[\B_+(\O_{\P(\clog{X}{D}}(1))=\bigcup_{\substack{I\subset\{1,\cdots,c\} \\ |I|<n}}\P\left(\O_{D_I}^{\oplus |I|}\right).\]
\end{Def}
A weaker property is used instead in the papers \cite{DR20} and \cite{ADT22}.
\begin{Def}
The logarithmic cotangent bundle $\clog{X}{D}$ is \emph{ample modulo} $D$ if it satisfies
\[p(\B_+(\O_{\P(\clog{X}{D}}(1)))=D,\]where $p:\P(\clog{X}{D})\to X$ is the canonical projection.
\end{Def}
Note that both properties imply Brody-hyperbolicity of the complement $X\setminus D$.

The very first result in this direction is due to Noguchi, in the case of general arrangements of lines in $\P^2$. It can be formulated with our vocabulary as follows.
\begin{Prop}[\cite{No86}]
The logarithmic tangent bundle along an arrangement $D$ of $c\geq6$ lines in general position with respect to hyperplanes and quadrics is ample modulo $D$.
\end{Prop}
This has recently been generalized to the higher-dimensional case in \cite{DR20}.
\begin{Theo}[Theorem A, \cite{DR20}]
\label{thm:darondeau}
    The logarithmic cotangent bundle along an arrangement $\mathscr{A}$ of $c\geq\binom{n+2}{2}$ hyperplanes in $\P^n$ in general position with respect to hyperplanes and to quadrics is ample modulo $\mathscr{A}$.
\end{Theo}

On the other hand, it is known that the logarithmic cotangent bundle associated to a smooth pair $(X,D)$ is almost ample when the degree of the irreducible components of $D$ is sufficiently large.
\begin{Theo}[Brotbek-Deng '18 \cite{BD18}]
Let $X$ be a smooth projective variety of dimension $n$ and let $c\geq n$. Let $L$ be a very ample line bundle on $X$. For any $m\geq (4n)^{n+2}$ and for general hypersurfaces $H_1,\cdots,H_c\in|L^m|$,  the logarithmic cotangent bundle $\clog{X}{D}$ associated to the smooth pair $(X,D=\sum_{i=1}^cH_i)$ is almost ample.
\end{Theo}

More generally, we can ask similar questions in the orbifold case, which interpolates between the compact and the logarithmic cases. Following Campana (see \cite{Cam04}), a smooth orbifold is a pair $(X,\Delta)$ where $X$ is a smooth complex projective variety and $\Delta$ is a $\mathbb{Q}$-divisor with normal crossings support and coefficients in $\Q\cap [0,1]$. One can define a sheaf $\Omega^1(X,\Delta)$ of differential forms with `fractional' poles: these are actually holomorphic forms living on a suitable ramified covering of $X$.

In \cite{DR20}, the authors also derived an orbifold analogue of their main theorem \ref{thm:darondeau}.
\begin{Theo}[Theorem B in \cite{DR20}]
The orbifold cotangent bundle along an arrangement $D$ of $c\geq\binom{n+2}{2}$ hyperplanes in $\P^n$ in general position with respect to hyperplanes and quadrics, with multiplicities $\geq 2n+2$, is ample modulo boundary.
\end{Theo}

In this paper, pushing further the results of \cite{No86} and \cite{DR20}, we focus on complements of hyperplane arrangements in $\P^n$, $n\geq 2$.

Remark that hyperbolicity of hyperplane complements was already known for several decades. Namely, we have the following results.

\begin{Theo}[Bloch, Cartan, Green, Kobayashi (see \cite{Kob98})]
    The complement of $c\ge 2n+1$ hyperplanes in general position in $\P^n$ is hyperbolic and hyperbolically embedded.
\end{Theo}
\begin{Theo}[Zaidenberg \cite{Zai86}]
The complement of $c\le 2n$ hyperplanes in $\P^n$ is never hyperbolic.
\end{Theo}

\subsection{Results}
In the very specific case of hyperplane arrangements, we first show that the two notions of positivity for the logarithmic cotangent bundle are in fact equivalent.
\begin{theop}{A (Theorem 3.17)}
Let $D=\sum_{i=0}^{c-1}H_i$ be an arrangement of $c$ hyperplanes in general position in $\P^n$. Then $\clog{\P^n}{D}$ is ample modulo $D$ if, and only if, it is almost ample.
\end{theop}


We also improve the lower bound of Theorem \ref{thm:darondeau} on the number of components of a general arrangement.
\begin{theop}{B (Corollary 3.11)}
Let $D=\sum_{i=0}^{c-1}H_i$ be a hyperplane arrangement in general position in $\P^n$. Then the corresponding logarithmic cotangent bundle is \emph{\bf ample modulo boundary} if, and only if the $c$ hyperplanes impose at least $4n-2$ independent conditions on quadrics.
\end{theop}
In particular, one needs at least $4n-2$ components. Moreover, this theorem includes the fact that the genericity condition is optimal.
\newpage
In fact, we are able to describe explicitly the augmented base locus whenever it is not trivial. In particular, we derive
\begin{theop}{C (Theorem 3.21)}
    The logarithmic cotangent bundle associated to $c$ general hyperplanes is big if and only if $c\geq 2n+1$. 
\end{theop}

We also obtain an orbifold version of Theorem B, which also improves \cite{DR20}.
\begin{theop}{D (Theorem 4.7)}
Let $\Delta=\sum_{i=1}^c(1-1/m_i)H_i$ be an orbifold divisor in $\P^n$, where the components $H_i$ are hyperplanes in general position.
Assume that $m_i\geq 2n$ for all $i$. Then the orbifold cotangent bundle $\Omega(\P^n,\Delta)$ is ample modulo $D$ if, and only if, the $c$ hyperplanes impose at least $4n-2$ independent conditions on quadrics.
\end{theop}
When the logarithmic cotangent bundle is ample modulo $D$, it is not difficult to see that the complement $X\setminus D$ is hyperbolic in the sense of Brody. This comes from a classical fact concerning entire curves.
In our last part, we give some applications in hyperbolicity. The following one, which concerns hyperbolicity of a very specific type of Fermat complete intersections, is new to the best of our knowledge.

\begin{theop}{E (Theorem 5.8)}
The Fermat cover associated to an arrangement of \(d\geq4n-2\) hyperplanes in \(\P^ n\) in general position, which impose at least $4n-2$ independent conditions on quadrics, with ramification \(m\geq 2n\), is Kobayashi-hyperbolic.
\end{theop}

\subsection{Organization of the paper}
In section \ref{section:projective_geometry}, we briefly recall some notions of classical projective geometry as dual varieties and rational normal scrolls.

After this, in subsection \ref{subsection:log_pairs}, we first give some definition and basic results concerning log pairs. We provide in subsection \ref{subsection:gauss} an alternative description of the augmented base locus of a globally generated line bundle. We study more carefully in subsection \ref{subsection:hyperplane} the structure of the logarithmic cotangent bundle associated to a hyperplane arrangement. 
Subsections \ref{subsection:jumping_lines} and \ref{subsection:almost_ampleness} are then devoted to the proofs of theorems A and B. Lastly, in subsection \ref{subsection:bplus}, we infer a precise description of the augmented base locus when it is not minimal.

In section \ref{section:orbifolds}, we give a brief overview of Campana's theory of geometric orbifolds, and we prove theorem D, which is more or less an extension of theorem B in this context.

Finally, we present in section \ref{section:hyperbolicity} some applications of our results concerning hyperbolicity of orbifold pairs and Fermat-type complete intersections.

\section{Some projective geometry}
\label{section:projective_geometry}
In this section, we discuss topics of classical algebraic geometry that will be used thereafter.
\subsection{Varieties of linear spaces}
Let $\G(k,n)=\grass(k+1,n+1)$ be the Grassmannian of $k-$linear subspaces in $\P^n$. This is a subvariety of $\P(\bigwedge^{k+1}\C^{n+1})$ of dimension $(k+1)(n-k)$.
We define a subvariety $\Sigma\subset\G(k,n)\times\P^n$ by setting
\[\Sigma=\{(Z,p); p\in Z\}.\]
Let $\pi_1,\pi_2$ be the projection maps from $\Sigma$ to $\G(k,n)$ and $\P^n$ respectively.
We can make use of $\Sigma$ to construct subvarieties of $\P^n$ swept out by $k$-planes.
\begin{Prop}[see \emph{e.g.} \cite{Har92}]
    Let $\Phi\subset\G(k,n)$ be any subvariety. Then 
    \[X=\bigcup_{Z\in\Phi}Z=\pi_2(\pi_1^{-1}(\Phi))\]
    is a subvariety of $\P^n$.
\end{Prop}

\subsubsection{Rational normal scrolls}
\label{ex_scroll}
A particular case of this construction is obtained when $\Phi\cong\P^1$.
Let $k\geq2$. A subvariety $X\subset\P^n$ given as the union 
\[X=\bigcup_{\lambda\in \P^1}Z_\lambda\]
of a 1-parameter family of $(k-1)-$planes is named a \emph{rational normal }$k$\emph{-fold scroll}. 

One can show (\cite{Har92}, Theorem 19.9) that $X$ has minimal degree $1+n-k$.

In particular, a hypersurface scroll $X$ is a quadric. According to \cite{Har92}, Ex 8.36, $X$ can also be characterized equivalently as:
    \begin{itemize}
        \item a quadric of rank 3 or 4, in the sense that the symmetric matrix representing $X$ has rank 3 or 4;
        \item a cone over a plane conic or a smooth quadric surface in $\P^3$;
        \item an irreducible quadric containing an $(n-2)$-plane.
    \end{itemize}

\subsection{Dual varieties}
Consider the $n$-dimensional projective space $\P^n$. The \emph{dual projective space} $(\P^n)^\ast$ is defined as the projective space parameterising hyperplanes in $\P^n$. 

Let $X$ be a subvariety of $\P^n$ of dimension $k$. We define the dual variety $X^\ast\subset(\P^n)^\ast$ to be the closure of the set of all hyperplanes tangent to $X$ at a smooth point. 

Restricting first to the case where $X$ is a smooth hypersurface, we may also say that $X^\ast$ is the locus of hyperplanes $H\subset\P^n$ such that the intersection $H\cap X$ is singular. It is indeed a variety since it is the image of $X$ under the \emph{Gauss map} $\mathscr{G}_X$, a morphism sending a point $p\in X$ to its embedded tangent hyperplane $\mathbb{T}_pX$; in coordinates, if $f$ is a homogeneous polynomial defining $X$, the Gauss morphism is given by
\[\mathscr{G}_X:p\mapsto\left[\frac{\dr F}{\dr Z_0}(p):\cdots:\frac{\dr F}{\dr Z_n}(p)\right].\]
If $X\subset \P^n$ is a \emph{smooth} hypersurface of degree $d\geq 2$, the morphism $\mathscr{G}_X$ is birational onto its image as well as finite. The dual $X^\ast$ is then again a hypersurface, though in general highly singular (never smooth if $d\geq3$). One can compute its degree $\deg X^\ast= d(d-1)^{n-1}$. In particular, the dual of a smooth quadric hypersurface is again a smooth quadric.

In the general situation, we define the \emph{conormal variety} as the closure $CX$ of the incidence correspondence
\[\{(p,H)\in\P^n\times(\P^n)^\ast|p\in X_{sm}\mbox{ and }\mathbb{T}_pX\subset H\}.\]The dual variety $X^\ast$ is then naturally defined as the image $\pi_2(CX)\subset\P^{n\ast}$ of the second projection.

A fundamental result concerning dual varieties is the reflexivity theorem.
\begin{Theo}
    Let $X\subset \P^n$ be an irreducible variety and $X^\ast\subset\P^{n\ast}$ its dual. Then the conormal variety $CX\subset \P^n\times\P^{n\ast}$ is equal to $CX^\ast\subset\P^{n\ast}\times\P^n\cong\P^n\times\P^{n\ast}$. It follows that $(X^\ast)^\ast=X$: the dual of a variety is the variety itself.
\end{Theo}

Considering a (possibly singular) subvariety $X\subset \P^n$, what can the dimension of $X^\ast$ be? One computes easily the dimension of $CX$: the fibre of the first projection map $CX\to X$ over a smooth point $p\in X_{sm}$ is simply the subspace $\P^{n-k-1}\subset\P^{n\ast}$ of hyperplanes containing the $k-$plane$\mathbb{T}_pX$. Hence the inverse image in $CX$ of the smooth locus of $X$ is irreducible of dimension $n-1$. We conclude that $X^\ast$ is irreducible of dimension less than $n-1$, and that it should be $(n-1)-$ dimensional in most cases. More precisely, $X^\ast$ has dimension $n-1$ exactly when the general tangent hyperplane to $X$ is tangent at only finitely many points. 

\begin{Ex}[see \cite{Har92} \cite{Do12}]
\label{dim_dual_scroll}
A situation where $X^\ast$ fails to be a hypersurface is when the tangent hyperplanes are constant along subvarieties of $X$: for example, when we have a variety ruled by linear subspaces, as cones over subvarieties of $\P^n$. Typical examples are given by rational normal scrolls of dimension $k\geq 3$. Recall that a $k$-dimensional scroll $X=\bigcup_{\lambda\in\P^1}Z_\lambda$ is a variety ruled in $\P^{k-1}$'s. The tangent plane $\mathbb{T}_pX$ at any point $p\in X$ will naturally contain the $(k-1)$-plane $Z_\lambda$ of the ruling through $p$.
We conclude that the general fibre of the second projection $CX\to\P^{n\ast}$ is $(k-2)$-dimensional, so that $X^*$ has dimension $n-k+1$.

For instance, if $X$ is a hypersurface scroll in $\P^n$, or in other words a quadric hypersurface of rank $r\geq4$, then its dual $X^\ast$ is a surface inside a $\P^3$ in the non-degenerate case ($r=4$) and a plane conic curve otherwise.
\end{Ex}
\section{Positivity of the logarithmic cotangent bundle}
\subsection{Conventions}
Let $X$ be a smooth complex projective variety.

For $D=\sum_i\alpha_iD_i$ an $\R-$divisor on $X$, we write as $\lceil D\rceil=\sum_iD_i$ the associated reduced divisor and $|D|$ its support.

Given a vector bundle $E$ on $X$, we denote by $\P(E)$ the bundle of rank one quotients of $E$ and by $\mc{O}_{\P(E)}(1)$ the tautological line bundle. We often identify $\P(E)$ with the projective space of lines $P(E^{\vee})$ of the dual of $E$.

If $L$ is a line bundle over $X$, its \emph{stable base locus} is $\mathbb{B}(L)=\bigcap_{m\in\N}\mathrm{Bs}(L^{\otimes m})$. For $m$ large and divisible enough, we have $\mathbb{B}(L)=\mathrm{Bs}(L^{\otimes m})$. 

We define its \emph{augmented base locus} (see e.g. \cite{Laz04-2}, Definition 10.3.2) as \[\mathbb{B}_+(L):=\bigcap_{p\in\N}\mathbb{B}(L^{\otimes p}\otimes A^{-1})\]
where $A$ is any ample divisor on $X$. Clearly $L$ is ample if and only if $\mathbb{B}_+(L)=\emptyset$, and $L$ is big if and only if $\mathbb{B}_+(L)\neq X$. 

In this last case, its complement is the largest Zariski-open subset $U\subset X\setminus\mathbb{B}(L)$ such that the restriction of the  map \[\Phi_{L^m}|_U:U\to\P(H^0(X,L^{\otimes m}))\] is an isomorphism onto its image, for all $m\gg1$ divisible enough.

For further details on augmented base loci, we refer to Lazarsfeld's book \emph{Positivity in algebraic geometry}, \cite{Laz04-2}.

\subsection{Augmented base loci of globally generated line bundles}

\label{subsection:gauss}
Let $L\to X$ be a holomorphic line bundle on a smooth projective variety $X$.
The evaluation map \[H^0(X,L)\otimes\mc{O}_X\longrightarrow L\] induces a rational map \[\Phi_L:X \dashrightarrow\P H^0(X,L).\]
By definition, $L$ is ample (resp. big) if and only if $\Phi_{L^m}$ is an embedding (resp. is birational onto its image) for some $m>0$. If $L$ is globally generated, $\Phi_L$ is a morphism.

Recall a theorem of Nakamaye concerning the augmented base locus of nef line bundles.
  \begin{Theo}[Nakamaye \cite{Nak00}. See also \cite{Laz04-2}, 10.3]
    Let $L$ be a nef line bundle on a smooth projective variety $X$. Then 
    \[\B_+(L)=\bigcup_{\substack{Z\subset X \\ Z\cdot L^{\dim Z}=0}}Z.\]
    \end{Theo}
It implies the following.
\begin{lemme}\label{fibres1}
    Let $Y$ be a projective variety and consider a regular morphism $f:X\to Y$. For any ample line bundle $A$ on $Y$, the augmented base locus of $f^*A$ is the union of the positive-dimensional fibres of $f$, \emph{i.e.}
    \[\B_+(f^*A)=\{x\in X ; \dim_x f^{-1}(\{f(x)\})>0\}=\colon \mathrm{NFL}(f).\]
\end{lemme}
\begin{proof}
\normalsize
  By Nakamaye's theorem,
  \[\B_+(f^*A)=\bigcup_{\substack{Z\subset X \\ Z\cdot f^*A^{\dim Z}=0}}Z.\]
  If $Z\subset X$ is not contained in $\mathrm{NFL}(f)$, then the restriction $f|_Z$ is generically finite. But this implies that
  \[Z\cdot f^*A^{\dim Z}=(\deg f|_Z)f(Z)\cdot A^{\dim Z}>0.\]
  The reverse inclusion is straightforward since by definition \[f(\mathrm{NFL}(f))\cdot A^{\dim \mathrm{NFL}(f)}=0.\]
\end{proof}

In particular, we have the following characterization of the augmented base locus of a globally generated line bundle.

\begin{Prop}
\label{fibres2}
 With the above notations, suppose that $L$ is globally generated. Then the augmented base locus $\B_+(L)$ is the union of all positive-dimensional fibres of $\Phi_L$, \emph{i.e.}
 \[\B_+(L)=\{w\in X;\dim\Phi_L^{-1}(\Phi_L(w))>0\}.\]
\end{Prop}
\begin{proof}
\normalsize
Since $L$ is globally generated, $\Phi_L$ is a well-defined morphism on $X$ and $L=\Phi_L^*\mc{O}_{\P H^0(X,L)}(1)$. So we can apply lemma \ref{fibres1} below with $A=\mc{O}_{\P H^0(X,L)}(1)$.
\end{proof}
\subsection{Positivity of logarithmic pairs}
\label{subsection:log_pairs}
Let $(X,D)$ be a smooth log pair, that is, $X$ is a smooth projective variety and $D$ is a simple normal crossings divisor.
The logarithmic tangent bundle $T_X(-\log D)$ is defined as the locally free subsheaf of $T_X$ of vector fields tangent to $D$. If $(z_1,\dots,z_n)$ are local coordinates on an open set $U\subset X$ such that $U\cap D=(z_1\dots z_r=0)$, then $T_X(-\log D)|_U$ is generated by \[z_1\frac{\partial}{\partial z_1},\dots,z_r\frac{\partial}{\partial z_r},\frac{\partial}{\partial z_{r+1}},\dots,\frac{\partial}{\partial z_n}.\] We now define the logarithmic cotangent bundle $\Omega_X(\log D)$ as the dual bundle of $T_X(-\log D)$. Concretely, it is the $\mc{O}_X-$module of meromorphic 1-forms with at most logarithmic poles, locally generated by the forms \[\frac{dz_1}{z_1},\dots,\frac{dz_r}{z_r}, dz_{r+1},\dots,dz_n.\]

Let $D=\sum_{1\leq i\leq c}D_i$ be the decomposition of $D$ into irreducible components.
Then the logarithmic cotangent bundle fits in the residue exact sequence \[0\longrightarrow\Omega_X\longrightarrow\clog{X}{D}\stackrel{\mathrm{Res}}{\longrightarrow}\bigoplus_{1\le i\le c}\mc{O}_{D_i}\longrightarrow0.\]Here the residue morphism is given over any open subset $U\subset X$ by \[\mathrm{Res}\left(\alpha+\sum_{i=1}^c\beta_i\frac{d\sigma_i}{\sigma_i}\right)=(\beta_i|_{D_i})_{1\leq i\leq c},\]where $\sigma_i$ is a section defining $D_i$ over $U$.

We have several other residue exact sequences corresponding to the various subsets of components of $D$, for any $I\subset\{1,\dots,c\}$ with $|I|<n$:
\[0\longrightarrow\Omega_X(\log D(I^\complement))\longrightarrow\clog{X}{D}\stackrel{\mathrm{Res}}{\longrightarrow}\bigoplus_{i\in I}\mc{O}_{D_i}\longrightarrow0,\]
where we denote $D(I^\complement)=\sum_{i\in I^\complement}D_i$.
The restriction of the above sequence to $D_I=\bigcap_{i\in I}D_i$ leads to a quotient bundle \[\clog{X}{D}|_{D_I}\stackrel{\mathrm{Res}}{\twoheadrightarrow}\mc{O}_{D_I}^{\oplus |I|}\]which induces an inclusion \[\widetilde{D_I}:=\P(\mc{O}^{\oplus r}_{D_I})\cong D_I\times\P^{|I|-1}\hookrightarrow\P(\clog{X}{D}).\]
Note that $\dim \widetilde{D_I}=\dim D_I+|I|-1=n-1$. 
When $X$ is projective of dimension $n$ at least 2 and
the number of components of $D$ is positive (i.e. $|D|=\emptyset$), as soon as $\dim D_I=n-|I|>0$, the trivial quotient of $\clog{X}{D}|_{D_I}$ prevents the logarithmic cotangent bundle from being ample. Said differently, it implies that for any such $I$,
\[\widetilde{D_I}\subset\B_+(\mc O_{\P(\clog{X}{D})}(1)).\]

A natural question to ask then is whether the $\widetilde {D_I}$ the only obstructions to the ampleness of the logarithmic cotangent bundle. This motivates the following definition.
\begin{Def}[Brotbek-Deng '18\cite{BD18}]
We say that the pair $(X,D)$ has \emph{\bf almost ample logarithmic cotangent bundle} if \[\mathbb{B}_+(\mc{O}_{\P(\clog{X}{D})}(1))\subset\bigcup_{|I|\leq n-1}\widetilde{D_I}.\]
\end{Def}
We will also be interested in the following \emph{a priori} weaker notion.
\begin{Def}[Darondeau-Rousseau '24 \cite{DR20}]
Let $p:\P(\clog{X}{D})\to X$ be the natural projection. The logarithmic cotangent bundle is said \emph{ample modulo boundary} if $p(\B_+(\mc{O}_{\P(\clog{X}{D})}(1)))=D$.
\end{Def}

\subsection{The case of hyperplane arrangements}
\label{subsection:hyperplane}
From now on, we will focus on the special case of hyperplane arrangements in $\P^n$. We start with two observations.

\begin{lemme}[\cite{BD18}, Prop. 4.2]
\label{irregularity}
If $D=\sum_ {i=1}^cD_i$ is an arrangement of $c$ hypersurfaces in general position in $\P^n$, then the associated logarithmic cotangent bundle has irregularity $h^0(\P^n,\clog{\P^n}{D})=c-1$.
\end{lemme}
\begin{proof}
\normalsize
For $1\leq i\leq c$, let $s_i\in H^0(\P^n,\mc{O}_{\P^n}(\lambda_i))$ be a section defining the hypersurfaces $D_i$. We claim that the forms $\omega_i=\lambda_1d\log s_i-\lambda_id\log s_1$ form a basis of global sections of $\clog{\P^n}{D}$.
First, these are indeed well-defined global sections. If we denote respectively by $x_1,\dots,x_n$ and $y_1,\dots,y_n$ the standard affine coordinates over the open subsets $U_0=\{Z_0\neq0\}$ and $U_1=\{Z_1\neq 0\}$, we have over $U_0\cap U_1$
\begin{align*}
\lambda_1d\log s_i(x)-\lambda_id\log s_1(x)&=\lambda_1d\log (x_1^{\lambda_1}s_i(y))-\lambda_id\log (x_1^{\lambda_1}s_1(y))\\ &=\lambda_1d\log s_i(y)-\lambda_id\log s_1(y).
\end{align*}
The residue exact sequence induces a long exact sequence in cohomology
\[0\longrightarrow H^0(\P^n,\clog{\P^n}{D})\stackrel{\Res}{\longrightarrow} \bigoplus_{i=1}^cH^0(D_i,\mc{O}_{D_i})\cong\C^c\stackrel{\delta}{\longrightarrow}H^1(\P^n,\Omega_{\P^n}).\]If we denote by $e_i\in H^0(D_i,\mc{O}_{D_i})$ the constant section equal to $1$, we compute $\mathrm{Res}(\omega_i)=\lambda_1e_i-\lambda_ie_1$.
As a consequence, we get $\lambda_1\delta(e_i)=\lambda_i\delta(e_1)$. In particular, the map $\delta$ has rank 1, so that the dimension of $H^0(\P^n,\clog{\P^n}{D})$ is exactly $c-1$.
\end{proof}
In particular, when the $D_i$'s are hyperplanes defined by linear forms $\ell_i$, a basis of global sections is given by the $1$-forms $\frac{d\ell_i}{\ell_i}-\frac{d\ell_1}{\ell_1}$.

It is now easy to see whether the associated logarithmic cotangent bundle is globally generated.
\begin{lemme}
 The logarithmic cotangent bundle associated with an arrangement of $c\geq n+1$ hyperplanes $H_j$ in general position in $\P^n$ is globally generated.
\end{lemme}
\begin{proof}
\normalsize
Without loss of generality, we can assume that $H_j=(Z_j=0)$ for $j=0,\dots,n$.
 The forms $\omega_j=\frac{dZ_j}{Z_j}-\frac{dZ_0}{Z_0}$ are global sections of $\clog{\P^n}{D}$. Hence a point $z$ in the base locus of $\clog{\P^n}{D}$ would satisfy $Z_j\xi_0-Z_0\xi_j=0$ for every $\xi\in\P^{n}$, which is impossible.
\end{proof}

From now on, we fix an arrangement $D=\sum_{i=0}^{n+k}H_i$ of $n+k+1\geq n+2$ hyperplanes in general position in $\P^n$. Without loss of generality, we assume that the first $n+1$ components are the coordinate hyperplanes. For $i=1,\dots,k$, let $\ell_i=a_0^iZ_0+\cdots+a_n^iZ_n$ be a linear form defining the hyperplane $H_{n+i}$.

In view of proposition \ref{fibres2}, we have to compute the fibres of the map $\Phi:=\Phi_{\O_{\P(\clog{\P^n}{D})}(1)}$ in order to study the augmented base locus of $\clog{\P^n}{D}$. To do so, we describe explicitly the morphism $\Phi$ in local coordinates.

Let us work in the affine chart $U_0=\{Z_0\neq0\}\subset\P^n$ equipped with standard coordinates $z_1,\dots,z_n$. By the proof of lemma \ref{irregularity}, a basis of global sections of $\clog{\P^n}{D}$ is given by \[\frac{dz_1}{z_1},\dots,\frac{dz_n}{z_n},\frac{d\ell_1}{\ell_1},\dots,\frac{d\ell_k}{\ell_k}.\]On the open set $U=U_0\cap\bigcap_{j=1,\dots,k}(\ell_j\neq0)$, a holomorphic frame for $T_{\P^n}(-\log D)$ is given by $z_1\frac{\dr}{\dr z_1},\dots,z_n\frac{\dr}{\dr z_n}$. Therefore, we may identify $\P(\clog{\P^n}{D}|_U)$ with $U\times\P^{n-1}$ via
\[\begin{array}{rcl}
U\times\P^{n-1}&\longleftrightarrow& \P(\clog{\P^n}{D}|_U)   \\
(z,[\xi])&\longmapsto&\left(z,\left[\xi_1z_1\frac{\dr}{\dr z_1}+\cdots+\xi_nz_n\frac{\dr}{\dr z_n}\right]\right)\end{array}\]
Since $\clog{\P^n}{D}$ is globally generated with irregularity $h^0(\P^n,\clog{\P^n}{D})=n+k$, we have a quotient
\[H^0(\P^n,\clog{\P^n}{D})\otimes\O_{\P^n}\to \clog{\P^n}{D}\to0\]inducing a map
\[\Phi:\P(\clog{\P^n}{D})\to\P(H^0(\P^n,\clog{\P^n}{D})\otimes\O_{\P^n})\simeq\P^n\times\P^{n+k-1}\twoheadrightarrow\P^{n+k-1}.\]

The morphism $\Phi$ restricted to $U$ then has the following expression:
\begin{align*}
\Phi(z,[\xi_1:\cdots:\xi_n])&=\left[\frac{dz_1}{z_1}\cdot\xi:\cdots:\frac{dz_n}{z_n}\cdot\xi:\frac{d\ell_1}{\ell_1}\cdot\xi:\cdots:\frac{d\ell_k}{\ell_k}\cdot\xi\right]\\
&=\left[\xi_1:\cdots:\xi_n:\frac{\sum_{j=1}^na_j^1\xi_jz_j}{a_0^1+\sum_{j=1}^na_j^1z_j}:\cdots:\frac{\sum_{j=1}^na_j^k\xi_jz_j}{a_0^k+\sum_{j=1}^na_j^kz_j}\right].
\end{align*}
From this we can obtain an explicit description of the fibres of $\Phi|_U$.
Take a vector $V=[V_1:\cdots:V_{n+k}]\in\P^{n+k-1}$. The fibre of $\Phi|_{U}$ above $V$ is described by the equations

\begin{equation}\tag{$\dagger$}\label{equations_fibre}\Phi(z,[\xi])=V\Longleftrightarrow \left\{\begin{array}{rclr}\xi_i&=&\lambda V_i, &i=1,\dots,n\\\sum_{i=0}^na_i^j(V_i-V_{n+j})z_i&=&0,&j=1,\dots,k\end{array}\right.\mbox {for some } \lambda\in\C\setminus\{0\}.\end{equation}
Therefore
\[\pi\left(\Phi^{-1}(V)\right)\cap U=\left\{(z_1,\dots,z_n);\sum_{i=0}^na_i^j(V_i-V_{n+j})z_i=0 \mbox{ for all }j=1,\dots,k\right\},\]
which is a linear subspace of $\P^n$ (for convenience, we let here $z_0=1$ and $\xi_0=V_0=0$).

Note that we can work out the same proof for any open subset $\P^n\setminus\bigcup_{i\notin I}H_i$, $|I|=n$, up to a linear coordinate change. Hence the result holds on the whole of $\P^n$.

Together with the results of \ref{subsection:gauss}, we have therefore proven that:
\begin{Propperso}
 The logarithmic cotangent bundle $\clog{\P^n}{D}$ fails to be ample modulo $D$ if and only if the projection of the augmented base locus $\B_+(\O_{\P(\clog{\P^n}{D})}(1))$ contains a line not in $\supp D$.

In other words, $\clog{\P^n}{D}$ is ample modulo $D$ if and only if its restriction to any line $l\not\subset \supp D$ is ample.
\end{Propperso}

The existence of a solution $(z,[\xi])$ to the equations \ref{equations_fibre} is equivalent to the matrix
\[\begin{pmatrix}
    a_{0}^1(V_{n+1}-V_0) & \cdots & a_{n}^1(V_{n+1}-V_n)\\
    \vdots&\ddots&\vdots\\
    a_{0}^k(V_{n+k}-V_0) & \cdots & a_{n}^k(V_{n+k}-V_n)
\end{pmatrix},\] being of rank $<n+1$.

Hence the image $W$ of $\Phi$ is described by the vanishing of the maximal minors of $M$.
Moreover, these computations characterize the positive dimensional fibres of $\Phi$ as the points $V$ where $\rg M<n$.

Hence, the logarithmic cotangent bundle is never big if $k\le n-1$.
On the other hand, by the general position assumption, the matrix $M$ is generically of maximal rank when $k\ge n$. In other words, $\clog{\P^n}{D}$ is big if and only if $n+k+1\ge 2n+1$.

\subsection{Construction of obstruction lines}
\label{subsection:jumping_lines}
In this section, we will use some notions and results from \cite{DK93}.

As in subsection \ref{subsection:hyperplane}, we consider a logarithmic pair $(\P^n,D)$, where $D=H_0+\cdots+H_{n+k}$ is an arrangement of $n+k+1\geq n+2$ hyperplanes in general position. We assume without loss of generality that the $n+1$ first components are the coordinates hyperplanes, i.e. $H_i=(Z_i=0)$ for $i=0,\dots,n$. We aim to prove the following result.
\begin{Theoperso}
\label{ampleness}
The logarithmic cotangent bundle of the pair $(\P^n,D)$ \textbf{fails to be ample modulo }$D$ if, and only if, there exists a line $l\not \subset D$ such that $l^\ast\cong\P^{n-2}$ and the points $H_i^\ast$ are contained in a same quadric hypersurface in $\P^{n\ast}$.
\end{Theoperso}

Let us first explain how this condition is translated in the original space $\P^n$.
Let $X\subset\P^{n\ast}$ denote a quadric hypersurface containing an $(n-2)$-plane. We have seen in Example \ref{dim_dual_scroll} that $X$ is a quadric of rank $\leq 4$, swept out by a 1-parameter family of $(n-2)$-planes: $X=\bigcup_{\lambda\in\P^1}Z_\lambda$. The dual variety $X^\ast$ is in general a non-degenerate quadric surface contained in a $\P^3\subset\P^n$ (when $X$ has rank $\leq3$, $X$ is a cone over a plane conic, and $X^\ast$ is also a cone over the dual conic curve). It is actually a ruled surface $\P^1\times\P^1$, whose two rulings correspond dually to the two rulings in $\P^{n-2}$'s of $X$.

\begin{rk}
The condition that the hyperplanes $H_i$ belong to a same quadric rational scroll $X$ in $\P^{n\ast}$ means that the $H_i$ are all tangent to the dual quadric surface $X^\ast\subset\P^n$.
\end{rk}
For sake of simplicity, considering $n+k+1$ hyperplanes $H_0,\dots,H_{n+k}$ in general position in $\P^n$, we introduce the following property. 
\begin{mdframed}
\begin{cond}[Condition $\color{red}\bigstar$]
~

The points $H_0^\ast,\dots,H_{n+k}^\ast$ in the dual projective space $\P^{n\ast}$ corresponding to the hyperplanes $H_0,\dots,H_{n+k}$ do not belong to a same quadric hypersurface of rank at most $4$.
\end{cond}
\end{mdframed}
Theorem \ref{ampleness} asserts that this is equivalent to $\clog{\P^n}{D}$ being ample modulo $D$.
\begin{corperso}
\label{main_thm}
The logarithmic cotangent bundle of the pair $(\P^n,D)$ is ample modulo $D$ if, and only if, the hyperplanes $H_i$ impose at least $4n-2$ linearly independent conditions on quadrics.
In particular, if $D$ has no more than $4n-3$ components, there always exist obstructions to ampleness away from $D$.
\end{corperso}
Note that for $n=2$, we retrieve Noguchi's result with six lines on $\P^n$. For $n=3$, we obtain the same bound as in \cite{DR20}, but as soon as $n\ge4$, our result is strictly better.

\begin{proof}
\normalsize
Following the previous section, the strategy of the proof relies on the search for lines contracted by the map $\Phi$.
Since the logarithmic cotangent bundle is globally generated, $\Phi$ is a morphism and an embedding when restricted to each fibre of the projection $\P(\clog{\P^n}{D})\to\P^n$.

According to the proposition \ref{construction} below, there exists a line $l\not\subset \supp D$ contracted by $\Phi$ if, and only if, there exists a degree 1 rational map $l\to H_0$ preserving each component $H_i$. Applying further \ref{quadric}, producing such a map amounts to the fact that the points of $(\P^n)^\ast$ corresponding to the hyperplanes $H_i$ lie on a same quadric $X$ of rank $\leq 4$. 

Moreover, we can compute the dimension of the space of quadrics of rank $\le4$ on $\P^n$. Choosing such a form amounts to taking a 4-codimensional subspace of $\C^{n+1}$ together with a quadratic form on $\C^4$. We find that the dimension equals $4(n-3)+\frac{4\times5}{2}-1=4n-3$.

This ends the proof of both the theorem and its corollary.
\end{proof}

We shall now prove in two main steps why the existence of such obstruction lines away from $\supp D$ is equivalent to the above condition concerning quadrics.

Let us first give some terminology.
\begin{Def}[\cite{DK93}]
A line $l\subset\P^n$ is called superjumping for the pair $(\P^n,D)$ if the restriction $\clog{\P^n}{D}|_l$ is not ample.
\end{Def}
Using the residue exact sequence, it is clear that all the lines belonging to $\supp D$ are superjumping.

In other words, the logarithmic cotangent bundle $\clog{\P^n}{D}$ is ample modulo boundary if and only if there is no superjumping line away from $\supp D$. The following characterization of superjumping lines is straightforward.
\begin{lemmeperso}
\label{trivial_quotient}
    A line $l\subset\P^n$ is superjumping if, and only if, the restriction $T_{\P^n}(-\log D)|_l$ of the logarithmic tangent bundle has a non-vanishing global section.
\end{lemmeperso}
\begin{proof}
\normalsize
   Recall that a vector bundle on $\P^1$ can be decomposed as a direct sum of line bundles. Since $\clog{\P^n}{D}|_l$ is globally generated, there exist non-negative integers $a_1\le \cdots\le a_n$ such that \[\clog{\P^n}{D}|_l=\mc{O}_l(a_1)\oplus\cdots\oplus\mc{O}_l(a_n),\]so that \[T_{\P^n}(-\log D)|_l=\mc{O}_l(-a_1)\oplus\cdots\oplus\mc{O}_l(-a_n).\]By definition, the line $l$ is superjumping if and only if $a_1=0$, \emph{i.e.} if $T_{\P^n}(-\log D)|_l$ has a trivial line sub-bundle.
\end{proof}

The next proposition is the key argument in the proof of the theorem. Note that even though this results appears as Prop. 7.5 in \cite{DK93}, our proof is completely different and self-contained. Ours is geometric and constructs explicitly the map whereas the authors of \cite{DK93} use algebraic constructions called \emph{elementary transformations}, in order to reconstruct inductively the logarithmic cotangent bundle, adding each component of $D$ one by one.
\begin{Propperso}
\label{construction}
Let $l$ be a line not lying in any hyperplane $H_i$. Denote by $p_i$ the point corresponding to $H_i$ in the dual projective space $(\P^n)^\ast$, and by $Z$ the 2-codimensional linear subspace of $(\P^n)^\ast$ corresponding to $l$. Then the following are equivalent:
\begin{enumerate}
    \item there exists a non-zero (even non-vanishing) global section of the bundle $T_{\P^n}(-\log D)|_l$;
    \item there exists a regular map $\psi:l\to H_0$ of degree 1 such that for all $1\le i\le n+k$, $\psi(l\cap H_i)\subset H_0\cap H_i$.
\end{enumerate}
\end{Propperso}

\begin{proof}
\normalsize
Let us construct in coordinates the desired map.
Let $X_0,\dots,X_n$ be homogeneous coordinates on $\P^n$ such that $H_0=(X_0=0)$ and $l=(X_2=\cdots=X_n=0)$.
On the open subset $U_0=\P^n\setminus H_0$, consider the inhomogeneous coordinates $(z_1,\dots,z_n)=(X_1/X_0,\dots,X_n/X_0)$. Then $z_1$ defines a coordinate on $l\setminus H_0\simeq \C$.

Assume first that there exists $\xi\in H^0(l,T_{\P^n}(-\log D)|_l)\setminus\{0\}$.
 We can write \[\xi=\xi_1(z_1)\frac{\dr}{\dr z_1}+\cdots+\xi_n(z_1)\frac{\dr}{\dr z_n}\]for some holomorphic functions $\xi_j:\C\to\C$.

Define now inhomogeneous coordinates on $U_1=\P^n\setminus \{X_1=0\}$ by \[(x_0,x_2,\dots,x_n)=(X_0/X_1,X_2/X_1,\dots,X_n/X_1)=(1/z_1,z_2/z_1,\dots,z_n/z_1).\]
The Jacobian matrix associated to this change of coordinates in $U_0\cap U_1$ is given by \[\left(\begin{array}{cccc}
    -1/z_1^2 & 0 &\cdots&0 \\
    -z_2/z_1^2 & 1/z_1&\ddots&\vdots\\
    \vdots &&\ddots&0\\
    -z_n/z_1^2&\cdots&\cdots&1/z_1
\end{array}\right).\]
On $l\cap U_0\cap U_1$, the vector field $\xi$ writes as
\[\xi(1/x_0)=-x_0^2\xi_1(1/x_0)\frac{\dr}{\dr x_0}+x_0\xi_2(1/x_0)\frac{\dr}{\dr x_2}+\dots+x_0\xi_n(1/x_0)\frac{\dr}{\dr x_n}.\]
Since $\xi$ needs to be holomorphic, the functions $x_0^2\xi_1(1/x_0),x_0\xi_2(1/x_0),\dots,x_0\xi_n(1/x_0)$ are holomorphic in the variable $x_0$. This shows that the holomorphic functions $x_j:\C\to\C$ are in fact polynomials, of degree $\leq 2$ for $j=1$ and $\leq 1$ otherwise.
Moreover, $\xi\in H^0(l,T_{\P^n}(-\log H_0)|_l)$, so we can write, for some holomorphic function $g$,
\[-x_0^2\xi_1\left(\frac{1}{x_0}\right)=x_0g(x_0).\]It follows that $\xi_1$ also has degree $\leq 1$.

There is a natural map $\psi:l\to H_0$ defined by 
\[\psi(z)=[0:\xi_1(z):\cdots:\xi_n(z)].\] 
By lemma \ref{trivial_quotient}, a global section of $T_{\P^n}(-\log D)$ comes from the trivial factors of the decomposition, so vanishes nowhere. Hence $\psi$ is well-defined.
By definition, $\xi$ is tangent to every component of $D$; this directly implies that $\psi$ sends each $H_j$ into itself.

Conversely, assume that we are given a degree 1 map $\psi:l\to H_0$ as above. Over $U_0=\P^n\setminus H_0$, write $\psi=[0:\xi_1:\cdots:\xi_n]$, and choose representatives $\xi_j$, which are degree 1 polynomials in the variable $z_1$. Then we can define a non-trivial logarithmic vector field by \[\xi=\sum_{j=1}^n\xi_j(z_1)\frac{\dr}{\dr z_j}\in H^0(l\setminus H_0,T_{\P^n}(-\log H_0)|_l).\]
By construction, one has $\xi(l\cap H_j)\in T_{l\cap H_j}H_j$ for all $j$, hence \[\xi \in H^0(l\setminus H_0,T_{\P^n}(-\log D)|_l).\]
\end{proof}

\clearpage
More concretely, given a global logarithmic vector field $\xi\in H^0(l,T_{\P^n}(-\log D)|_l)$, we construct the map $\psi:l\to H_0$ in the following way (see figures \ref{fig1} and \ref{fig:p3} for illustrations of the 2- and 3-dimensional cases): starting from a point $q$ in $l$, we have a tangent vector to $l$ at $x$, which uniquely defines a line $L_x$ through $x$. Denote by $\psi(x)$ the intersection point of $L_x$ with $H_0$. By definition, the sections of $T_{\P^n}(-\log D)|_l$ are those vector fields tangent to all $H_i$'s. Hence, if $x\in H_i\cap l$, the point $\psi(x)$ is also in $H_i$.
\vspace{3\baselineskip}
\begin{figure}[h!]
\centering

    \includegraphics[width=9cm]{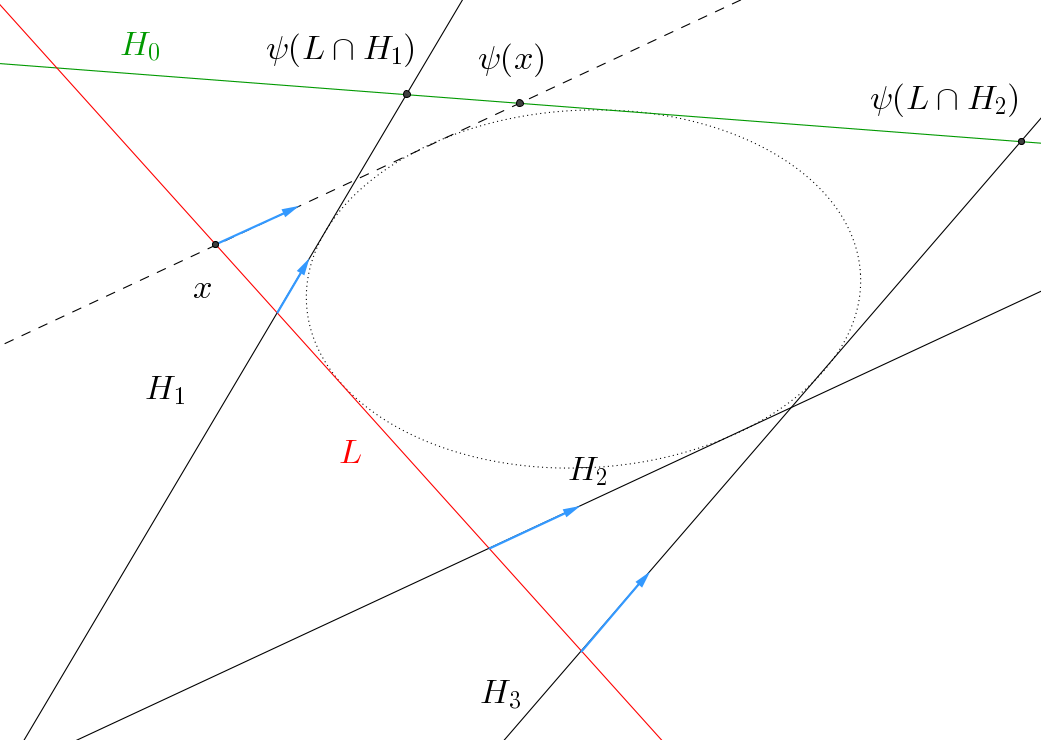}
    \caption{Lines in $\P^2$}
    \label{fig1}
    \end{figure}
\begin{figure}[h!]
\centering
    \includegraphics[width=12cm]{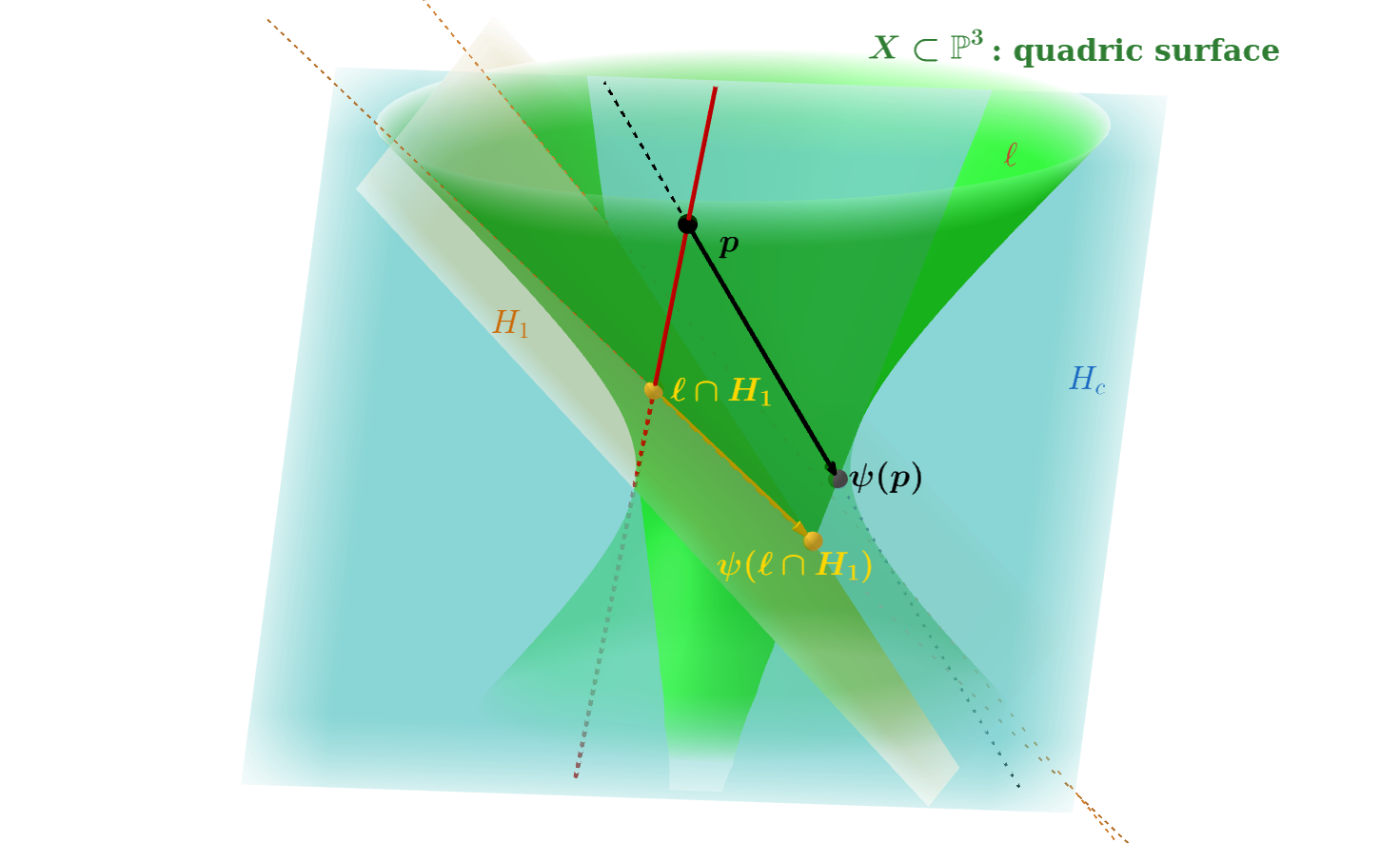}
    \caption{Construction for planes in $\P^3$}
    \label{fig:p3}

\end{figure}

Let $p_0,\dots,p_{n+k}$ be the points of $(\P^n)^\ast$ corresponding dually to the hyperplanes $H_0,\dots,H_{n+k}$ and let $W=l^\ast\subset (\P^n)^\ast$ be the $(n-2)$-plane of hyperplanes in $\P^n$ containing $l$. Assume that there exists a quadric hypersurface $X\subset (\P^n)^\ast$ containing $W,p_0,\dots,p_{n+k}$.

The inclusion $W\cup\{p_0,\dots,p_c\}\subset X$ means that the line $l$ and all the hyperplanes $H_i$ are tangent to the dual surface $X^\ast$ at some points.

The next proposition provides a necessary and sufficient condition for the existence of such a quadric hypersurface.
\begin{Propperso}
\label{quadric}

Consider $c+1$ points $p,p_1,\dots,p_c$ in general position in $\P^n$, and a $(n-2)$-plane $W$ disjoint from $\{p_0,p_1,\dots,p_c\}$. For $i=1,\dots,c$, denote by $L_i=\langle W,p_i\rangle$ the hyperplane generated by $W$ and $p_i$. The following two assertions are equivalent.
\begin{enumerate}
\item There exists a quadric hypersurface scroll $X\subset\P^n$ containing $W$ and all the points $p_i$.
\item There exists a degree 1 map $\psi:\P^1\cong  W^\ast \to \P^{n-1}\cong p_0^\ast$ such that $\psi(L_i)$ contains the point $p_i$ and $\psi(L_0)\neq L_0.$
\end{enumerate}
\end{Propperso}

\begin{proof}
\normalsize

\begin{enumerate}

\item Consider a degree 1 map $\psi$ as described in the statement. 
Then $\psi$ sends $W^\ast\cong\P^1$ to a line in $\P^{n-1}\cong p_0^\ast$.

The intersection $W'=\bigcap_{L\in W^\ast}\psi(L)$ is therefore another $(n-2)$-plane.

Define a subset of $\P^n$ by \[X(\psi)=\bigcup_{L\in W^\ast}L\cap\psi(L).\] It contains $W$ as well as the points $p_i$. Indeed, for any $q\in W\setminus W'$, there exists a unique hyperplane $\psi(L)$ containing both $W'$ and $q$. It follows that $q\in L\cap\psi(L)\subset X(\psi)$. Moreover, by definition of $\psi$, we know that $p_i\in L_i\cap\psi(L_i)$ for all $i$, so that $p_i\in X(\psi)$.

By construction, $X(\psi)$ is an $(n-1)$-dimensional scroll of degree 2 in $\P^n$, which contains $W\cup\{p_0,\dots,p_c\}$.

\item Assume the existence of the quadric $X$. We will show that $X$ is necessarily of the form $X(\psi)$ for some map $\psi$.

The set $W^\ast$ of hyperplanes containing $W$ is a projective line $\P^1$. Let $L_{\infty}\in W^\ast $ be the tangent hyperplane to $X$ at any point of $W$. Since $p_0\in X\setminus W$, $p_0\not\in L_\infty$.

For any $L\in W^\ast $, the set $L\cap X$ is a quadric hypersurface inside $L$, which contains $W$. Hence there exists an $(n-2)$-plane $W(L)$ such that $L\cap X=W\cup W(L)$. One has $W(L)\neq W$ except for $L=L_\infty$. Moreover, $p_0\in W(L)$ only if $L=L_0$.

We define a map $\psi:W^\ast\to p_0^\ast$ by setting $\psi(L)=<W(L),p_0>$ for $L\neq L_0$, and $\psi(L_0)=\mathbb{T}_{p_0}X$ (embedded tangent hyperplane to $X$ at $p_0$). Note that if since $p_i\in X\setminus W$, $p_i\in W(L_i)$. Then $\psi$ has the desired properties, with $\psi(L_\infty)=L_0$.

Moreover, for any hyperplane $L$, one has $Z(L)=L\cap \psi(L)$, and $X$ is described as the union 
\[\bigcup_{L\in W^\ast}L\cap \psi(L).\]

\end{enumerate}
\end{proof}

\begin{rk}
    We can also understand this construction in the original space $\P^n$.

    Consider a ruled quadric surface $X\cong\P^1\times\P^1$ containing the line $l$ and tangent to all hyperplanes $H_i$.

A point $p\in l$ belongs to two lines of $X$, one in each ruling. Let $l(p)\subset X$ be the other line through $p$. The map $\psi$ sends $p$ to the intersection point of $l(p)$ with $H_0$.
\end{rk}
\subsection{Almost ampleness}
\label{subsection:almost_ampleness}
Let $(X,D)$ be a smooth logarithmic pair. Whereas being ample modulo $D$ only involves the image of $\B_+(\O_{\P(\clog{X}{D})}(1))$ in $X$, the stronger notion of almost ampleness reads on the augmented base locus itself.
In general, this last property is strictly stronger. However, in the particular case of hyperplane arrangements, both are equivalent.
\begin{Theoperso}
Let $D=\sum_{i=0}^{n+k}D_i$ be a arrangement of hyperplanes in general position in $\P^n$. Then $\clog{\P^n}{D}$ is ample modulo $D$ if, and only if, it is almost ample.
\end{Theoperso}

\begin{proof}
\normalsize
We only need to treat one direction. Assume that the logarithmic cotangent bundle associated to $D$ is ample modulo $D$. Then $\B_+(\O_{\P(\clog{X}{D})}(1))\subset \pi^{-1}(D)$.

We will study more carefully the fibres of $\Phi$ above the strata of $D$. According to lemma \ref{fibres2}, their union forms the augmented base locus $\B_+(\O_{\P(\clog{X}{D})}(1))$. 

We need to show that $\Phi$ restricted to $\pi^{-1}(D)\setminus \bigcup_{I;|I|<n}\widetilde{D}_I$ is injective.

Let $I\subset\{0,\dots,n+k\}$ be a subset of indices of cardinality $<n$. Up to a linear coordinate change, we can assume that $I=\{1,\dots,r\}$ with $r\leq n$.

Again, let us work on the open subset $U=(Z_0\neq0)\cap\bigcap_{1\leq j\leq k}(\ell_j\neq0)$ with standard affine coordinates $(z_1,\dots,z_n)$. In these coordinates, $D_I=\{(0,\dots,0,z_{r+1},\dots,z_n)\}\cong\P^{n-r-1}$.
Moreover, the restricted residue map $\Res:\clog{\P^n}{D}|_{D_I\cap U}\to\mc{O}_{D_I}^{\oplus r}$ is given by the expression
\[\Res|_{D_I\cap U}\left(\sum_{i=1}^n\eta_i\frac{dz_i}{z_i}\right)=(\eta_1,\dots,\eta_r).\]
Under the above trivialization, $\widetilde{D}_I$ is described as \[\widetilde{D}_I\cap\pi^{-1}(U)\cong\left\{\begin{array}{c|c}&z_i=0~\forall ~1\leq i\leq r\\(z,[\xi_1:\cdots:\xi_n])\in U\times\P^{n-1}&\mbox{and }\xi_i=0\\&\forall i\geq r+1\end{array}\right\}.\]

Remark first that for any $J\subset\{0,\dots,n+k\}$, $\pi^{-1}(D_I)\cap\widetilde{D}_J=\emptyset$ unless $J\subset I$. This can be easily checked in coordinates.

For all $(z,[\xi])\in\pi^{-1}(D_I)$, we have 
\begin{equation}
\label{expr_phi}\tag{$\blacklozenge$}
\Phi(z,[\xi])=\left[\xi_1:\cdots:\xi_n:\frac{\sum_{j={\color{red} r+1}}^na_j^1\xi_jz_j}{a_0^1+\sum_{j={\color{red} r+1}}^na_j^1z_j}:\cdots:\frac{\sum_{j={\color{red} r+1}}^na_j^k\xi_jz_j}{a_0^k+\sum_{j={\color{red} r+1}}^na_j^kz_j}\right].
\end{equation}

Note that for $J\subset I$ and $(z,[\xi]) \in \pi^{-1}(D_I)\cap\widetilde{D}_J$, one has \[\Phi(z,[\xi])=[\un_J(1)\xi_1:\cdots:\un_J(r)\xi_r:0:\cdots:0].\] 
In particular, this image is independent of the $z$-coordinates and $\Phi^{-1}(\Phi(z,[\xi]))\cong D_J$.

The restriction $\Phi|_{\pi^{-1}(D_I)}$ is injective out of $\bigcup_{J\subset I}\widetilde{D}_J$ if and only if the corresponding map 
    \[\Phi^\dag:\C^{n-r}\times(\C^n\setminus(\C^r\times\{0\}^{n-r}))\longrightarrow\C^r\times\C^{n-r}\times\C^{k}\]
is so. 

Since the first coordinate map of $\Phi^\dag$ is given by the identity $\C^r\to\C^r$, this is equivalent to the injectivity of 
\begin{align*}
\label{phidag}\tag{$\spadesuit$}
\Phi^\ddag :\C^{n-r}\times(\C^{n-r}\setminus\{0\})&\longrightarrow\C^{n-r+k}\\
(z,\xi)&\longmapsto\left(\xi_{r+1},\dots,\xi_n,\frac{\sum_{j=r+1}^na_j^1\xi_jz_j}{a_0^1+\sum_{j=r+1}^na_j^1z_j},\dots,\frac{\sum_{j=r+1}^na_j^k\xi_jz_j}{a_0^k+\sum_{j=r+1}^na_j^kz_j}\right).
\end{align*}

We see that $\Phi|_{\pi^{-1}(D_I)}$ is injective out of $\bigcup_{J\subset I}\widetilde{D}_J$ whenever the above map $\Phi^\ddag:\C^{n-r}\times\C^{n-r}\to\C^{n-r+k}$ is injective out of the corresponding subset. We recognize in equation \ref{phidag} the expression for the morphism associated to the pair $(D_I,D(I^\complement)|_{D_I})$. Hence we are reduced to showing that this arrangement of $(n-r)+k+1$ hyperplanes in $D_I\cong \P^{n-r}$ has almost ample logarithmic cotangent bundle.

By lemma \ref{subarrangement} below, we know that this pair already has ample modulo boundary logarithmic cotangent bundle. 

The conclusion follows by induction on the dimension, provided that we prove the case $n=2$. 

For any $(z=(0,z_2),[\xi_1:\xi_2])\in\pi^{-1}(D_1\cap U_0)$, \ref{expr_phi} reads as
\[\Phi(z,[\xi])=\left[\xi_1:\xi_2:\frac{a^1_2\xi_2z_2}{a_0^1+a_2^1z_2}:\cdots\right].\]
A quick computation shows that this map is injective in restriction to \[\pi^{-1}(D_1\cap U_0)\setminus\widetilde{D}_1=\{((0,z_2),[\xi_1:1])\}.\]
\end{proof}
\begin{lemmeperso}
\label{subarrangement}
  Let $D=\sum_{0\leq i\leq n+k}D_i$ be an arrangement of $n+k+1$ hyperplanes in general position in \(\P^n\) satisfying Condition $\color{red}\bigstar$. Then so does the arrangement $D(I^{\complement})|_{D_I}$ in $D_I\cong \P^{n-|I|}$, for any subset of indices $I$ of cardinality $<n$.

\end{lemmeperso}
\begin{proof}
\normalsize
We assume for simplicity that $I=\{0,\dots,r\}$ and $D_i=(Z_i=0)$ for $i=0,\dots,r$.

First, it is clear that the arrangement $D(I^{\complement})|_{D_I}$ is in general position (\emph{i.e.} has normal crossings singularities).

Assume that the hyperplanes $D_j\cap D_I$, $i=r+1,\cdots,n+k$ do not satisfy Condition $\color{red}\bigstar$. Then one can find a symmetric matrix $Q'$ of size $n-r$ and rank less than $4$ such that
\[A_I^\intercal Q' A_I=0,\]then the matrix \[Q=\left(\begin{array}{c|c}0_{r+1}&0\\ \hline 0&Q'\end{array}\right)\] satisfies \[A^\intercal Q A=0\]and Condition $\color{red}\bigstar$ does not hold for the hyperplanes $D_i$'s in $\P^n$.

\end{proof}

\subsection{Description of the augmented base locus in terms of the number of hyperplanes}
\label{subsection:bplus}
The proof of theorem \ref{ampleness} actually provides a complete description of the image  
\[p(\B_+(\O_{\P(\clog{\P^n}{D})}(1)))\]
of the augmented base locus of the logarithmic cotangent bundle $\clog{\P^n}{D}$. 
\begin{Theoperso}
    The projection
    \[p\left(\B_+(\O_{\P(\clog{\P^n}{D})}(1))\right)\]is the union of all ruled quadric surfaces which are simultaneously tangent to all hyperplanes $H_i$.
    Hence, for $4n-3\geq c\geq3n$ hyperplanes, this locus has dimension $4n-1-c$, and
     \[p\left(\B_+(\O_{\P(\clog{\P^n}{D})}(1))\right)=\P^n\]
     if $c\leq 3n-1$.
\end{Theoperso}
\begin{proof}
\normalsize
\begin{itemize}[label=$\bullet$]
\item We know indeed that the locus $p(\B_+(\O_{\P(\clog{\P^n}{D})}(1)))$ is a union of lines. According to the construction of the previous section, the $(n-2)-$plane, dual of any such line, is contained in some quadric scroll in $\P^{n\ast}$, which also contains the hyperplanes $H_i$'s as points of the dual projective space. Conversely, for any quadric scroll $X$ containing the $H_i$'s in $\P^{n\ast}$, the line dual to any $\P^{n-2}$ of its rulings is a superjumping line. When $n\geq 3$, it is a line of the quadric surface dual to $X$. 

Consequently, the image of $\B_+(\O_{\clog{\P^n}{D}}(1)$ in $\P^n$ is made of $\supp D$ and all the quadric surfaces that are simultaneously tangent to all the components $H_i$.

\item Recall that the dimension of the space of quadrics of rank $\leq4$ in $\P^{n\ast}$ equals $4n-3$. In the limit case $c=4n-3$, for a general choice of hyperplanes (here, the term \emph{general} is to be understood as \emph{imposing linearly independent conditions on quadrics}), there exists a unique quadric $X$ of rank exactly 4 containing all of them in $\P^{n\ast}$. Thus $\B_+(\O_{\clog{\P^n}{D}}(1)$ projects onto $D\cup S$, where $S=X^\ast$ is a smooth quadric surface. 

Similarly, the variety of rank $\leq 4$ quadrics containing $c$ general points in $\P^n$ has dimension $4n-3-c$. We derive that for a general choice of $c$ hyperplanes, $p(\B_+(\O_{\P(\clog{\P^n}{D})}(1)))\setminus \supp D$ has dimension $2+4n-3-c=4n-1-c$. For $c<3n$, this augmented base locus is not any more dominant.

\item We deal separately with the 2-dimensional case, since we do not get surfaces. The superjumping lines showing up in the proof are the points of any conic in $\P^{2\ast}$ passing through all the points $H_i$. In the original plane $\P^2$, if there exists a conic $\mc{C}$ osculating the lines $H_i$, then any other tangent line is superjumping. We see that these lines shape exactly $\P^2=p(\B_+(\O_{\P(\clog{\P^2}{D})}(1)))$.

\item For $n\ge3$, denote by $\Sigma$ the closure
\[\overline{p(\B_+(\O_{\P(\clog{\P^n}{D})}(1)))\setminus \supp D}\]
of the complement of $\supp D$ in the projection of the augmented base locus of the logarithmic cotangent bundle.
We have seen that $\Sigma$ consists of a union of lines. In addition, the superjumping lines were identified, through the construction of the previous section, as the lines of all quadric surfaces $S$ contained in some $\P^3\subset \P^n$, which are simultaneously tangent to all components of $D$. 

It is known that any irreducible quadric surface is ruled (it isomorphic to $\P^1\times\P^1$). Any surface $S$ as above is thus entirely contained in $\Sigma$.
We have ultimately described the locus $\Sigma$ as a union of quadric surfaces.

\item To be more precise, we will determine the dimension of the family of such surfaces $S$, depending on the number of components of $D$.
Recall from the previous sections that such a surface $S$ is the dual variety of a quadric hypersurface in $\P^{n\ast}$ of rank $\leq4$ containing all the points $p_i=H_i^\ast$.

First, let us compute the dimension of the variety $\Xi$ of quadric hypersurfaces of rank $\leq4$ in $\P^n$.
To do so, we use the characterization of quadrics of rank at most 4 as quadric hypersurfaces containing an $(n-2)$-plane.
The variety $\Xi$ can then be represented as the image under the second projection of the incidence variety 
\[\mathscr{I}=\{(Z,X);Z\subset X\}\subset\G(n-2,n)\times\P^{\binom{n+2}{2}-1}.\]

Fix an $(n-2)$-plane $Z\subset\P^n$. What can be the dimension of the fibre $\pi_1^{-1}(Z)$ of the first projection above $Z$? Observe that this fibre is actually the kernel of the linear map
\[\C^{\binom{n+2}{2}}\cong H^0(\P^n,\O_{\P^n}(2))\twoheadrightarrow H^0(Z,\O_{Z}(2))\cong\C^{\binom{n}{2}}.\]
Hence $\pi_1^{-1}(Z)$ has dimension $\binom{n+2}{2}-\binom{n}{2}-1=2n$.

In the end, we find $\dim\mathscr{I}=2n+\dim\G(n-2,n)=2n+2(n-1)=4n-2$. Since the fibre above a general point of $\Xi$ has dimension 1, we compute $\dim\Xi=4n-3$.

We see now that for a general choice of $c\geq 4n-2$ hyperplanes in $\P^n$, the image $p(\B_+(\O_{\P(\clog{\P^n}{D})}(1)))$ of the augmented base locus is empty, since $4n-2$ general points can not belong to a same quadric of rank $\leq4$.
If $D$ has $c\leq 4n-3$ general components, then the locus $\Sigma$ is swept out by a $(4n-3-c)$-dimensional family of quadric surfaces, and thus has dimension $4n-3-c+2=4n-1-c$ for $c\geq3n$. When we have strictly less than $3n$ components, $\Sigma$ is the whole space $\P^n$.
\end{itemize}
\end{proof}

When $c\leq 3n-1$, $\P^n$ is entirely covered by superjumping lines. However, the augmented base locus $\B_+(\mc{O}_{\P(\clog{\P^n}{D})}(1))$ upstairs can still be a proper subset of $\P(\clog{\P^n}{D})$. Is it possible to describe it more precisely?

Assume that $2n\leq c\leq 3n-1$. The image $p(\B_+(\mc{O}(1))$ is swept out by a $(4n-3-c)$-dimensional family of $\P^1\times\P^1$'s tangent to all the hyperplanes. Therefore, the set of superjumping lines is a subvariety of $\G(1,n)\cong\P^{2n-2}$ of dimension $4n-2-c$. We remark that for $c\leq 2n$, all lines in $\P^n$ are superjumping.
\begin{Theoperso}
    Assume that we have $n+2\leq c\leq 3n-1$ hyperplanes $H_i$ in general position in $\P^n$. Then the augmented base locus 
    \[\B_+(\O_{\P(\clog{\P^n}{D})}(1))\subset\P(\clog{\P^n}{D})\]
    has dimension $\leq4n-1-c$.
\end{Theoperso}
\begin{proof}
\normalsize
We first show that $\B_+(\O_{\P(\clog{\P^n}{D})}(1))$ equals exactly the union over all superjumping lines $l$ of the curves $\P(\O_l)$ induced by trivial quotients of $\clog{\P^n}{D}$.

Write $c=n+k+1$ with $k\ge1$ and let the notation be as in subsection \ref{subsection:hyperplane}. We work in the standard affine coordinates associated to the open subset $\P^n\setminus H_0=\{Z_0\neq 0\}$. 

Consider an element \[\left(x,\xi=\sum_{i=1}^n\xi_i\frac{\dr}{\dr z_i}\right)\in\B_+(\O_{\P(\clog{\P^n}{D})}(1))\setminus p^{-1}(D)\] and denote for all $i=1,\dots,n$ and $j=1,\dots,k$, \[V_i=\xi_i, V_j=\frac{\sum_{i=1}^na_i^jx_i\xi_i}{\sum_{i=0}^na_i^jx_i},\] so that $\Phi(x,\xi)=[V_1:\cdots:V_{n+k}]$. Then the fibre $Z=\Phi^{-1}(\Phi(x,\xi))$ is positive-dimensional and $x$ belongs to some superjumping line $l\subset p(Z)$. The inclusion $l\subset p(Z)$ shows that any $z=(z_1,\dots,z_n)\in l\setminus H_0$ satisfies the equations 
\[\sum_{i=0}^na_i^jz_i(V_{n+j}-V_i)=0.\]
One can extend the vector field 
\[\xi(z)=\sum_{i=1}^n\frac{\xi_i}{x_i}z_i\frac{\dr}{\dr z_i}\]
defined on $l\setminus H_0$ as a never-vanishing holomorphic global section of $T_{\P^n}(-\log D)|_l$. In other words, there is a trivial quotient $\O_l$ of $\clog{\P^n}{D}|_l$ such that $(x,\xi)\in\P(\O_l)$. 

In order to compute the dimension of $\B_+(\O_{\P(\clog{\P^n}{D})}(1))$, we will `count' all such curves $\P(\O_l)$.

Above each surface $\P^1\times\P^1$ composing $p(\B_+(\O_{\P(\clog{\P^n}{D})}(1)))$ lives a $1$-parameter family of $\P(\O_l)$ inside $\B_+(\O_{\P(\clog{\P^n}{D})}(1))$. Therefore, the augmented base locus has dimension at most $4n-2-c+1=4n-1-c$.
    
\end{proof}
If $c\leq 2n$, the morphism $\Phi:\P(\clog{\P^n}{D})\to\P^{c-2}$ cannot be generically finite, since $c-2\leq 2n-2<2n-1=\dim \P(\clog{\P^n}{D})$. We deduce
\begin{Theoperso}
    The logarithmic cotangent bundle associated to $c$ general hyperplanes is big if and only if $c\geq 2n+1$. 
\end{Theoperso}

\begin{rk}
    Note that $c=2n+1$ is also the bound for Brody-hyperbolicity of the complement $\P^n\setminus D$.
\end{rk}

\section{Orbifolds}
\label{section:orbifolds}
\subsection{Campana's orbifold category}
In what follows, the term \emph{orbifold} will exclusively correspond to geometric orbifold pairs as defined by Campana. We refer for the basic notions concerning orbifolds to the papers \cite{Cam11}, \cite{CP15} or \cite{CDR20}, for instance.

\begin{Def}
A smooth orbifold pair is a pair $(X,\Delta)$ where $X$ is a smooth projective variety and $\Delta$ is an effective $\Q$-divisor with coefficients in $[0,1]$ and simple normal crossings support. In analogy with ramification divisors, we will write it under the form $\Delta=\sum_{i\in I}(1-\frac{1}{m_i})D_i$ where the $D_i$ are prime irreducible divisors and $m_i\in\Q_{\geq 1}\cup\{\infty\}$. 

\end{Def}
In this work, we will only consider integer multiplicities $m_i$.
When the $m_i$ are all infinite (resp. equal to 1), we recover the classical logarithmic (resp. compact) case. One can then naturally regard orbifolds as an interpolation between the compact and the logarithmic cases.

As for the logarithmic case, we would like to define \emph{orbifold differential forms} as meromorphic differential forms over $X$ having singularities of order ``at most $1-1/m_i$ along each $D_i$''; formally, in some adapted local coordinates, the vector bundle $\Omega_{(X,\Delta)}$ of orbifold 1-forms should be generated over $\mc{O}_X$ by the forms $\frac{dz_i}{z_i^{1-1/m_i}}$.
Even though this is not directly possible, one can define these bundles through appropriate ramified coverings turning $\Delta$ into an integral divisor.
\begin{Def}
Let $Y$ be a smooth projective variety. A Galois covering $\pi:Y\to X$ is \emph{adapted} to the orbifold pair $(X,\Delta)$ if it satisfies the following conditions:
\begin{enumerate}
    \item for any irreducible component $D_i$, we have $\pi^*D_i=p_i\widetilde{D_i}$, where $p_i$ is a multiple of $m_i$ and $\widetilde{D_i}$ has at most simple normal crossings;
    \item both the supports of $\pi^*\Delta +\mathrm{Ram}(\pi)$ and the branch locus of $\pi$ have at most normal crossings.
\end{enumerate}
In addition, $\pi$ is said strictly adapted to $(X,\Delta)$ if $p_i=m_i$ for all $i$.
\end{Def}
For a smooth orbifold pair $(X,\Delta)$, there always exists an adapted covering (see Prop. 4.1.12 in \cite{Laz04-1}).

Let \(\pi:Y\to (X,\Delta)\) be an adapted covering. For any point \(y\in Y\), there exists an open neighbourhood $U\ni y$, invariant under the action of the isotropy group of \(y\) in \(\Aut(\pi)\). Hence, there exist local coordinates \(w_{i}\) on $U$ centred at $y$ such that \(\pi(U)\) has coordinates \(z_{i}\) centred at \(\pi(y)\) satisfying $|\Delta|\cap\pi(U)\subset\{z_1\dots z_n=0\}$ and 
\[
  \pi(w_{1},\dots,w_{n})
  =
  (z_{1}^{p_{1}},
  \dots,
  z_{n}^{p_{n}}),
\]
where \(p_{i}\) is an integer multiple of the multiplicity \(m_i\) in $\Delta$ of \((z_{i}=0)\).

If all multiplicities are infinite (\(\Delta=\lceil\Delta\rceil\)), for any adapted covering \(\pi\colon Y\to X\), we denote
\[
  \Omega(\pi,\Delta):=
  \pi^{\ast}\clog{X}{\Delta}.
\]
The orbifold cotangent bundle associated with $\pi$ is then defined as the locally free subsheaf $\Omega(\pi,\Delta)$ of $\Omega(\pi,\lceil\Delta\rceil)$ fitting in the short exact sequence
\begin{equation}\label{exact_orbifolds}0\rightarrow\Omega(\pi,\Delta)\rightarrow\Omega(\pi,\lceil\Delta\rceil)\rightarrow\bigoplus_{\begin{array}{c}i\in I\\m_i<\infty\end{array}}\mc{O}_{\pi^*D_i/m_i}\rightarrow0.\end{equation}

Here the quotient is the composition of the pullback of the residue map
\[
  \pi^*\Res : \pi^*\clog{X}{\lceil\Delta\rceil}\to \bigoplus_{\begin{array}{c}i\in I\\m_i<\infty\end{array}}
  \O_{\pi^*D_{i}}
\]
with the quotients
\(
\O_{\pi^*D_{i}}
\twoheadrightarrow
\O_{\pi^*D_{i}/m_{i}}
\).

The bundle $\Omega(\pi,\Delta)$ is locally generated in coordinates as above by the elements
\begin{equation}
\label{eq:orbifold_forms}\tag{$\maltese$}
  w_{i}^{\frac{p_{i}}{m_{i}}}
  \pi^*(d z_{i}/z_{i})
  =
  w_{i}^{-p_{i}(1-\frac{1}{m_{i}})}
  \pi^*(d z_{i}).
\end{equation}

We are now able to construct orbifold differential forms on $X$.
\begin{Def}
Given an adapted covering $\pi:Y\to(X,\Delta)$, the sheaf of orbifold symmetric differential forms of order $q$ is the direct image \[S^{[q]}\Omega_{(X,\Delta)}:=\pi_*\left((S^q\Omega(\pi,\Delta))^{\mathrm{Aut}(\pi)}\right)\subseteq S^q\clog{X}{\lceil\Delta\rceil}.\]
\end{Def}
One says that $\Omega_{(X,\Delta)}$ is big if $\Omega(\pi,\Delta)$ is a big vector bundle over $Y$ for some adapted covering $\pi:Y\to X$. This is equivalent to saying that for some (or any) ample line bundle over $X$, there exits an integer $N$ such that $H^0(X,S^{[N]}\Omega_{(X,\Delta)}\otimes A^{-1})\neq \{0\}$.
By definition, the \emph{augmented base locus} of $\Omega_{(X,\Delta)}$ is \[\mathbb{B}_+\left(\Omega_{(X,\Delta)}\right)=\bigcap_{N\geq 1}\bigcap_{p/q\in\Q}\bs\left(S^{[Nq]}\Omega_{(X,\Delta)}\otimes A^{-Np}\right)\]for some ample line bundle $A$ over $X$. Away from $\lceil\Delta\rceil$, this set turns out to be independent of the covering $\pi$.
\begin{Prop}[\cite{DR20}]
Over $X\backslash\lceil\Delta\rceil$, the image of the augmented base locus $\mathbb{B}_+\left(\mc{O}_{Y'}(1)\right)$ by the natural projection coincides with the orbifold augmented base locus $\mathbb{B}_+(\Omega_{(X,\Delta)}).$
\end{Prop}

As for the logarithmic case, it is important to note that the orbifold cotangent bundle cannot be ample when the support of $\Delta$ is non-empty.
\begin{lemme}[see \cite{DR20}, Lemma 3.7]
    Let $(\P^n,\Delta)$ be a smooth orbifold pair. Then for any strictly adapted covering $\pi:Y\to\P^n$, the orbifold cotangent bundle $\Omega_{(\P^n,\Delta)}$ has negative quotients supported on each component of $\Delta$ with finite multiplicity and trivial quotients supported on each component with infinite multiplicity.
\end{lemme}

This motivates the following definition.
\begin{Def}
    The orbifold cotangent bundle $\Omega_{(\P^n,\Delta)}$ is said \emph{ample modulo boundary} if $\mathbb{B}_+(\Omega_{(X,\Delta)})\subset\lceil\Delta\rceil$.
\end{Def}

\subsection{The Fermat cover for hyperplane arrangements}
Assume that the hyperplanes have equal orbifold multiplicities. In this particular case, we can construct a global strictly adapted covering as follows.

For any $i=0,\dots,n+k$, let $\ell_i$ be a linear form defining $H_i$. Without loss of generality, we may assume that $\ell_i(Z)=Z_i$ for $i_0,\dots,n$, so that $H_0,\dots,H_n$ are the coordinates hyperplanes, and we write \[\ell_{n+j}(Z)=\sum_{i=0}^na_{i}^jZ_i, j=1,\dots,k.\]In the projective space $\P^{n+k}$, one can identify $\P^n$ with the linear subspace \[\bigcap_{j=1}^{k}\left\{X_{n+j}=\sum_{i=1}^na_{i,j}X_i\right\}.\] 
Let $Y$ be the complete Fermat intersection \[\bigcap_{j=1}^{k}\left\{X_{n+j}^m=\sum_{i=1}^na_{i,j}X_i^m\right\}.\] This is an $n$-dimensional subvariety of $\P^{n+k}$. Then the holomorphic map 
\[ \pi\colon
\begin{array}{lrcl}
    &Y&\longrightarrow&\P^n\\
    &[X_0:\cdots:X_{n+k}]&\longmapsto&[X_0^m:\cdots:X_n^m]
\end{array}\]
realises $Y$ as a strictly adapted cover of the pair $(\P^n,\Delta)$, that is, a Galois covering ramifying exactly over the hyperplanes $H_i$ with ramification order $m$.
This is the \emph{Fermat cover} associated with the pair $(\P^n,\Delta)$.

\subsection{Positivity of the orbifold cotangent bundle}
In this context, we will prove the following analogue of \ref{main_thm}.
\begin{Theoperso}
\label{theo_orbifolds}
If the arrangement satisfies $\color{red}\bigstar$ and for all $i$, $m_i\geq 2n$, then the orbifold cotangent bundle $\Omega_{(\P^n,\Delta)}$ is ample modulo $|\Delta|$.

\end{Theoperso}

\begin{proof}
\normalsize
To begin with, increasing the multiplicities $m_i$ will always increase the ampleness of $\Omega(\P^n,\Delta)$ (see \cite{DR20}). Indeed, consider two orbifolds divisors with the same support \[\Delta=\sum_{i=0}^{n+k}(1-1/m_i)H_i,\Delta'=\sum_{i=0}^{n+k}(1-1/m_i')H_i\]such that for each $i$, one has $m_i\leq m_i'$. We can find a ramified covering $\pi:X\to \P^n$ which is adapted for both $\Delta$ and $\Delta'$. Namely, one has for any $i$, \[\pi^*D_i=p_im_iE_i=p_i'm_i'E_i.\]
Writing this in adapted local coordinates as in \ref{eq:orbifold_forms}, we see that $\Omega(\pi,\Delta)$ is a subsheaf of $\Omega(\pi,\Delta')$; similarly $S^{[q]}\Omega_{(\P^n,\Delta)}$ is a subsheaf of $S^{[q]}\Omega_{(\P^n,\Delta')}$. 
Therefore, \[\mathbb{B}_+(\Omega_{(X,\Delta')})\subset\mathbb{B}_+(\Omega_{(X,\Delta)}),\]and we can assume without loss of generality that all the $m_i$ are equal.

Recall that in subsection \ref{subsection:hyperplane}, we provided equations describing the image $W$ and the positive-dimensional fibres of the morphism $\Phi$ associated with $\clog{\P^n}{|\Delta|}$.

Denote by $p$ the projection $\P(\clog{\P^n}{|\Delta|})\to\P^n$. We will make use of equations \ref{equations_fibre} to construct global orbifold forms vanishing on an ample divisor, whose base locus is contained in \[p\left(\B_+(\mc{O}_{\P(\clog{\P^n}{D})}(1))\right).\]

The image $\Phi(p^{-1}(H_1))$ is described by the equations 
\begin{align*}
    \left\{\begin{array}{rcl}V_1&=&\xi_1\\&\vdots&\\V_n&=&\xi_n\\ (a_{0}^j+\sum_{i=2}^na_{i}^jz_i)V_{n+j}&=&\sum_{i=2}^na_{i}^jz_i\xi_i
    \end{array}\right.\Longleftrightarrow&
    \left\{\begin{array}{rcl}V_1&=&\xi_1\\&\vdots&\\V_n&=&\xi_n\\ a_{0}^jV_{n+j}+\sum_{i=2}^na_{i}^jz_i(V_{n+j}-V_i)&=&0
    \end{array}\right.\\
    &\\
    \Longleftrightarrow&\quad \ker M'\neq\{0\}\\
    &\\
    \Longleftrightarrow&\quad \rg(M')<n
\end{align*}
where $M'$ is the $n\times k$-matrix
\[\begin{pmatrix}
    a_0^1(V_{n+1}-V_0) &  a_{2}^1(V_{n+1}-V_2)&\cdots & a_{n}^1(V_{n+1}-V_n)\\
    \vdots&\ddots&\vdots\\
    a_0^k(V_{n+k}-V_0) & a_2^k(V_{n+k}-V_2)&\cdots & a_{n}^k(V_{n+k}-V_n)
\end{pmatrix}.\]
Again, a point $V\in\P^{n+k-1}$ belongs to $\Phi(\pi^{-1}(H_1))$ if all the $n\times n$-minors vanish at $V$.
Hence, any $n\times n$-minor $\Pi$ satisfies $\Phi^*\Pi|_{\pi^{-1}(H_1)}\equiv 0$, so that \[(p^*\ell_1)^{-1}\Phi^*\Pi\in H^0\left(\P(\clog{\P^n}{D}),\O_{\P(\clog{\P^n}{D})}(n)\otimes p^*\O_{\P^n}(-H_1)\right).\]

For instance, take for $\Pi$ the minor of $M'$ made of the first $n$ rows.

Then $(p^*\ell_1)^{-1}\Phi^*\Pi$ corresponds to a global symmetric form \[\omega\in H^0\left(\P^n,S^n\clog{\P^n}{D}\otimes \mc{O}_{\P^n}(-1)\right).\]The form $\omega$ has poles only along the hyperplanes $H_0,H_2,\dots,H_{2n}$. Moreover, each monomial in $\omega$ has at most a simple pole along each component $H_i$. 

Denote \[\eta=\ell_0\ell_2\cdots\ell_{2n}\omega^{2n}\in H^0(\P^n,S^{2n^2}\clog{\P^n}{D}).\] 
We will show that $\eta$ is an orbifold form as soon as $m\geq 2n$.

Indeed, recall the Fermat cover $\pi:Y\subset\P^{n+k}\to (\P^n,\Delta)$ given by \[\pi([X_0:\cdots:X_{n+k}])=[X_0^m:\cdots:X_n^m].\] In these coordinates, for all $i$, we have $\pi^*\ell_i=X_i^m$ and $\pi^*\frac{d\ell_i}{\ell_i}=m\frac{dX_i}{X_i}$. Hence \[\pi^*\eta=X_0^mX_2^m\cdots X_{2n}^m(\pi^*\omega)^{2n}.\]But we know that each monomial in $\omega^{2n}$ has pole order at most $2n$ along each component $H_i$, so that we can write \[\pi^*\eta=(X_0X_2\cdots X_{2n})^{m-2n}\widetilde{\omega}\in H^0\left(Y,S^{2n^2}\Omega_Y\right)\]for some holomorphic $\widetilde{\omega}$.
Therefore, \[\omega^{2n}=(\ell_0\ell_2\cdots\ell_{2n})^{-1}\eta\in H^0\left(\P^n,S^{[2n^2]}\Omega_{(\P^n,\Delta)}\otimes \O_{\P^n}(-2n)\right).\]

By the above procedure, up to reordering the hyperplanes $H_i$, we are able to construct, for any subset $I\subset \{0,\dots,n+k\}$ with $|I|=2n$, a global symmetric differential form $\omega_I\in H^0\left(\P^n,S^{[2n^2]}\Omega_{(\P^n,\Delta)}\otimes \O_{\P^n}(-2n)\right)$ with simple poles along each hyperplane $H_i, i\in I$.
Assume that all these sections $\omega_I$ vanish simultaneously at a point $(z,[\xi])\in\P(\clog{\P^n}{D})$. This means that all $n\times n$-minors of the $(n+1)\times(n+1)$-matrix 
\[M=\begin{pmatrix}
    a_0^1(V_{n+1}-V_0) &  a_{1}^1(V_{n+1}-V_1)&\cdots & a_{n}^1(V_{n+1}-V_n)\\
    \vdots&\ddots&\vdots\\
    a_0^k(V_{n+k}-V_0) & a_1^k(V_{n+k}-V_1)&\cdots & a_{n}^k(V_{n+k}-V_n)
\end{pmatrix}\]associated to the fibre of $\Phi$ above $V=\Phi(z,[\xi])$ vanish simultaneously. In other words, $M$ has rank at most $n$. As we have seen, this characterizes $\Phi^{-1}(\Phi(z,[\xi]))$ as a component of $\B_+(\O_{\P(\clog{\P^n}{D}}(1))$. 
We conclude that \[\bigcap_{|I|=2n}\bs(\omega_I)\subset p(\B_+(\O_{\P(\clog{\P^n}{D}}(1))).\]
\end{proof}
Note the difference between our orbifold forms and the ones of \cite{DR20}. While the latter are explicitly constructed in coordinates, ours have a geometric interpretation and directly come from some specific logarithmic forms.

\section{Applications}
\label{section:hyperbolicity}
In this section, we give some applications of our results to complex hyperbolicity. Note that part of the arguments in the proofs below were actually inspired by similar results in \cite{DR20}. The authors have chosen not to include them in the published version of the paper, but they can still be found in the ArXiv version \cite{DR20-preprint}. 
However, to the best of our knowledge, our improvements of these facts are genuinely new.

\subsection{An orbifold Brody theorem}
Let $(X,\Delta=\sum_i(1-1/m_i)\Delta_i)$ be an orbifold pair. It is natural to extend the notions of hyperbolicity to $(X,\Delta)$ using holomorphic morphisms $h:\mathbb{D}\to(X,\Delta)$ that are compatible with the orbifold structure. Namely, such $h$ shall satisfy two conditions:
\begin{itemize}
    \item $h(\mathbb{D})\not\subset|\Delta|$;
    \item for any $x\in h^{-1}(\Delta_i)$, one has $\mult_xh^*\Delta_i\ge m_i$.
\end{itemize}
Orbifold curves $f:\C\to(X,\Delta)$ are defined similarly.

One then defines the \emph{orbifold Kobayashi pseudo-distance} $d_{(X,\Delta)}$ as the largest pseudo-distance on $X\setminus |\Delta|$ such that any orbifold morphism $h:\mathbb{D}\to(X,\Delta)$ is distance-decreasing with respect to the Poincaré distance on the unit disk, \emph{i.e.} \[d_{(X,\Delta)}\leq h^*d_{\mathbb{D}}.\]

 The orbifold $(X,\Delta)$ is said \emph{Kobayashi-hyperbolic} if the pseudo-distance $d_{(X,\Delta)}$ is non-degenerate.
Likewise, $(X,\Delta)$ is \emph{Brody-hyperbolic} if there does not exist any non-constant orbifold curve $f:\C\to(X,\Delta)$.
As in the classical setting, Kobayashi-hyperbolicity implies Brody-hyperbolicity.

In \cite{CW09}, we can find the following generalization of Brody's reparametrisation lemma to orbifolds.
\begin{Theo}
 Let $(X,\Delta)$ be a compact orbifold. Assume that $(X,\Delta)$ is not hyperbolic, \emph{i.e.} $d_{(X,\Delta)}$ is not a distance. Then there exists a non-constant holomorphic map $f:\C\to X$ which either is an orbifold morphism or satisfies $f(\C)\subset|\Delta|$. Furthermore,
 \[\sup\Vert f'(z)\Vert =\Vert f'(0)\Vert >0.\]
\end{Theo}

As mentioned before, the following fact comes directly from the earlier version of \cite{DR20} (that can be found on the ArXiv).

\begin{Prop}
Let $(X,\Delta)$ be a smooth orbifold pair, with $\Delta=\sum_i(1-1/m_i)\Delta_i$. Assume that a sequence of orbifold maps $h_n:D\to(X,\Delta)$ form the unit disk to $(X,\Delta)$ converges locally uniformly to a holomorphic map $h:D\to X$. Let 
\[X_h=\bigcap_{h(D)\subset \Delta_i}\Delta_i, \Delta_h=\bigcap_{h(D)\not\subset\Delta_i}\Delta_i.\]
Then $h$ is an orbifold map $D\to (X_h,\Delta_h)$.
\end{Prop}

\begin{proof}
Suppose that $h(0) \in |\Delta|$. Consider a neighbourhood
$V$ of $h(0)$ in $X$ such that $|\Delta| \cap V$ is locally defined by a holomorphic function
$\prod_if_i$ , where $f_i = 0$ defines $\Delta_i \cap V$. If
$h(D) \not\subset \Delta_j$ , one can assume that $f_j \circ h$ has no zero in $V$ except at $0$.
Apply the classical theorem of Rouché to the sequence of holomorphic functions $(f_j \circ h_p)$.
For all sufficiently large $p$ the multiplicity at $0$ of $f_j \circ h$ equals the sum of all multiplicities
of all zeroes in $V$ of $f_j \circ h_p$. Therefore this multiplicity is at least $m_j$ because the $h_p$ are
orbifold maps.
\end{proof}
As an immediate consequence, reasoning exactly as in \cite{CW09}, we obtain
the following result.
\begin{Theo}
\label{th:orbifold_brody}
Consider a smooth orbifold pair $(X,\Delta)$ as above.
For a subset of indices $I$ , let $\Delta_I=\bigcap_{i\in I}\Delta_i$ , and let \[\Delta(I^\complement)=\sum_{j\notin I}(1-1/m_j)\Delta_j.\]
If all pairs $(\Delta_I,\Delta(I^\complement)|_{\Delta_I})$ are Brody-hyperbolic, then the pair $(X,\Delta)$ is Kobayashi-hyperbolic.
\end{Theo}
\subsection{Hyperbolicity of some orbifold pairs}
As in the classical setting, we can relate the existence of orbifold symmetric differentials to Brody-hyperbolicity, thanks to the following orbifold analogue of Theorem \ref{vanishing}.
\begin{Theo}[Fundamental vanishing theorem \cite{CDR20}]
\label{th:vanishing_orb}
 Let $(X,\Delta)$ be a smooth orbifold with $X$ projective. Fix an ample line bundle $A$ on $X$ and a global orbifold symmetric differential \[\omega\in H^0(X,S^{[N]}\Omega_{(X,\Delta)}\otimes A^{-1}).\] 
 Then for any orbifold entire curve $f:\C\to (X,\Delta)$, one has
 \[f^\ast\omega\equiv0.\]
\end{Theo}
In the previous section, we proved the inclusion 
 \[\mathbb{B}_+\left(\Omega_{(\P^n,\Delta)}\right)\subseteq p(\B_+(\mc{O}_{\P(\clog{\P^n}{D})}(1))).\]

We now infer from our Theorem \ref{theo_orbifolds} a result about algebraic degeneracy of orbifold curves.

\begin{corperso}
\label{th:orb-hyp}
    Consider the orbifold pair \[\left(\P^n,\Delta=\sum_{1\leq i\leq c}\left(1-\frac{1}{m_i}\right)H_i\right)\] formed by an arrangement of $c$ hyperplanes in general position with orbifold multiplicities $m_i\geq 2n$. If the arrangement satisfies $\color{red}\bigstar$, then the orbifold pair $(\P^n,\Delta)$ is Brody-hyperbolic. In fact, $(\P^n,\Delta)$ is even Kobayashi-hyperbolic.
\end{corperso}
This result is not new and follows from Nochka's Theorem below. Since our proof completely differs, though, we have chosen to include it.
\begin{Theo}[Nochka \cite{Noc83}]
Let $f:\C\to\P^n$ be a holomorphic map and let $d$ be the minimal integer such that the image of $f$ is contained in a $d$-dimensional subspace. Let $H_1,\dots,H_q$ be hyperplanes in general position in $\P^n$. Assume that the curve $f$ intersects each $H_i$ with multiplicity $m_i$. Then
\[\sum_{i=1}^q\left(1-\frac{d}{m_i}\right)<2n-d+1.\]
As a consequence, if $q\geq 2n+2$ and $m_i\geq 2n$, the orbifold pair $(\P^n,\Delta=\sum_{i=1}^q(1-1/m_i)H_i)$ is Brody-hyperbolic.
\end{Theo}
\begin{proof}
\normalsize
    According to Theorem \ref{th:orbifold_brody}, if $(\P^n,\Delta)$ is not hyperbolic, there exists either a non-constant orbifold curve $f:\C\to(\P^n,\Delta)$ or an orbifold curve inside a stratum $(\Delta_I,\Delta(I^\complement)|_{\Delta_I})$. Combining Theorems \ref{theo_orbifolds} and \ref{th:vanishing_orb}, all orbifold entire curves $\C\to(\P^n,\Delta)$ are constant. Hence we are left with the second possibility.
    But according to Lemma \ref{subarrangement}, $(\Delta_I,\Delta(I^\complement)|_{\Delta_I})$ is again an orbifold pair satisfying the conditions of Theorem \ref{theo_orbifolds}, so that it cannot contain any non-constant orbifold curve. 
\end{proof}

 \subsection{Hyperbolicity of Fermat covers}
Fermat hypersurfaces constitute one class of varieties for which several hyperbolicity results have been obtained. For instance, one has the following two results, due to M. Green.
\begin{Theo}[\cite{Kob98}, Ex. 3.10.21]
Let \[F(n,d)=\left\{[Z_0:\cdots:Z_{n+1}]; Z_0^d+\cdots+Z_{n+1}^d=0\right\}\subset\P^{n+1}\]
be the Fermat hypersurface of degree $d$ in $\P^{n+1}$.

\begin{enumerate}
\item If $d\geq (n+1)^2$, then every entire curve $f:\C\to F(n,d)$ has its image contained in a linear subspace of dimension $\lfloor n/2\rfloor$.
\item If $d>(n+1)(n+2)$, then every entire curve $f:\C\to\P^{n+1} \setminus F(n,d)$ has its image contained in a linear subspace of dimension $\lfloor (n+1)/2\rfloor$.
\end{enumerate}
\end{Theo}
These results are consequences of \emph{Cartan's truncated defect relation} (see \cite{Kob98}, 3.B.42) which, with the orbifold terminology, gives the linear degeneracy of orbifold curves inside an orbifold pair $(\P^n,\sum_{0\leq i\leq n+1}(1-1/m)H_i)$ attached to an arrangement of $n+2$ general hyperplanes, provided that $(n+2)(1-n/m)^+>n+1$. The hypersurface $F(n,m)$ is precisely the Fermat cover associated with this orbifold.

We use now Theorem \ref{theo_orbifolds} to prove the hyperbolicity of Fermat covers with different assumptions.
\begin{Theoperso}
    The Fermat cover associated with an arrangement of hyperplanes in $\P^n$ imposing at least $4n-2$ linearly independent conditions on quadrics, with ramification $m\geq 2n$, is Kobayashi-hyperbolic.
\end{Theoperso}
\begin{proof}
\normalsize
    Let $\pi:Y\to (\P^n,\Delta)$ be the Fermat cover. It is enough to prove that $Y$ is Brody-hyperbolic.

Let $f:\C\to Y$ be an entire curve. According to Theorem \ref{theo_orbifolds}, its image $f(\C)$ lies in the ramification locus of the covering $\pi$.

Note that the ramification locus can be seen as the Fermat cover associated to an arrangement of $d\geq 4n-2$ hyperplanes in $\P^{n-1}$. By lemma \ref{subarrangement} above, one can still assume that the genericity conditions of Theorem \ref{theo_orbifolds} are satisfied. Hence we obtain the hyperbolicity of $Y$.
\end{proof}
\section*{Acknowledgements}
The author is deeply grateful to her PhD supervisors Damian Brotbek and Damien Mégy for their precious help and advices, and their constant support during the past years.

She warmly thanks Julie Wang whose questions and suggestions contributed a lot to improve and clarify the corresponding part of her thesis manuscript.

She thanks Erwan Rousseau and Lionel Darondeau for the enlightening discussions they had around this work.

She also thanks Christophe Mourougane and Philippe Eyssidieux for the interest they showed about this work and their numerous suggestions for future work.

Lastly, she would like to thank the anonymous referee for helpful suggestions on the manuscript.

\bibliography{biblio.bib}{}
\bibliographystyle{siam}

\end{document}